\documentclass{amsart}  %<<<1
\usepackage{amssymb,amsmath,amscd,amsthm}

\theoremstyle{plain}
\newtheorem{theorem}{Theorem}
\newtheorem{lemma}[theorem]{Lemma}
\newtheorem{proposition}[theorem]{Proposition}

\theoremstyle{definition}

\theoremstyle{remark}
\newtheorem{remark}{Remark}

\newcommand{\lref}[1]{Lemma~\ref{#1}}
\newcommand{\pref}[1]{Proposition~\ref{#1}}
\newcommand{\rref}[1]{Remark~\ref{#1}}
\newcommand{\tableref}[1]{Table~\ref{#1}}

\newcommand{\gth}{\theta}
\newcommand{\Sph}{\mathbb{S}}
\newcommand{\Disc}{\mathbb{D}}

\newcommand{\HGR}{\ensuremath{\operatorname{H}}}
\newcommand{\GGR}{\ensuremath{\operatorname{G}}}
\newcommand{\SOG}{\ensuremath{\operatorname{SO}}}

\newcommand{\SpG}{\ensuremath{\operatorname{Sp}}}
\newcommand{\ULG}{\ensuremath{\operatorname{U}}}
\newcommand{\SUG}{\ensuremath{\operatorname{SU}}}
\newcommand{\Spin}{\ensuremath{\operatorname{Spin}}}

\newcommand{\KGR}{\ensuremath{\operatorname{K}}}
\newcommand{\trace}{\ensuremath{\operatorname{tr}}}
\newcommand{\diag}{\ensuremath{\operatorname{diag}}}
\newcommand{\Adj}{\ensuremath{\operatorname{Ad}}}
\DeclareMathOperator{\Ide}{Id}
\DeclareMathOperator{\spam}{span}

\newcommand{\noin}{\noindent}

\begin{document}

\title{Smoothness Conditions in Cohomogeneity one manifolds}

\author{L.~Verdiani}
\address {University of Firenze}
\email{luigi.verdiani@unifi.it}
\author{W.~Ziller}
\address{University of Pennsylvania}
\email{wziller@math.upenn.edu}
\thanks{The first named author was supported by PRIN and GNSAGA grants, the second named author was supported by a grant from the National Science Foundation}

\begin{abstract}
We present an efficient method for determining the conditions that a metric on a cohomogeneity one manifold, defined in terms of functions on the regular part, needs to satisfy in order to extend smoothly to the singular orbit.
\end{abstract}

\maketitle

%-------------- Article Text--------------------
\section*{Introduction}

A group action is called a cohomogeneity one action if its generic orbits are hypersurfaces. Such actions have been used frequently to construct examples of various types: Einstein metrics, soliton metrics, metrics with positive or non-negative curvature and metrics with special holonomy. See  \cite{4, 6, 7, 8, 12} for a selection of such results. The advantage of such metrics is that geometric problems are reduced to studying its behavior along a fixed geodesic $c(t)$ normal to all orbits.  The metric is  described by a finite collection of functions of $t$, which for each time specifies the homogeneous metric on the principal orbits. One aspect one needs to understand is what conditions these functions must satisfy if regular orbits collapse to a lower dimensional singular orbit. These smoothness conditions are often crucial ingredients in obstructions, e.g. to non-negative or positive curvature, see  \cite{9,14,15}.  The goal of this paper is to devise a simple procedure in order to derive such conditions explicitly.

\smallskip

The local structure of a cohomogeneity one manifold near a collapsing orbit can be described in terms of Lie subgroups  $H\subset K\subset G$  with $K/H=\Sph^\ell$, $ \ell>0$. The action of $K$ on $\Sph^\ell$ extends to a linear action on $\Disc=\Disc^{{\ell}+1} \subset \mathbb{R}^{\ell+1}$ and thus $M=G\times_K \Disc$ is a homogeneous disc bundle, where $K$ acts as $(g,p)\to (gk^{-1},kp)$, and with boundary $G\times_K\partial\, \Disc=G\times_KK/H=G/H$ a principal orbit.   The Lie group $G$ acts by cohomogeneity one on $M$ by left multiplication in the first coordinate.   A compact (simply connected) cohomogeneity one manifold is the union of two such homogeneous disc bundles. For simplicity we write $M=G\times_K V$ with $V\simeq\mathbb{R}^n$.  Given a smooth $G$ invariant   metric on the open dense set of regular points, i.e., the complement of the lower dimensional singular orbit,     the problem is when the extension of this metric to the singular orbit is smooth. We first simplify the problem as follows:

\begin{theorem}\label{2planes}
Let $G$ act by cohomogeneity one on $M=G\times_K V$ and $g$ be a smooth cohomogeneity one metric  defined on the set of regular points in $M$. Then $g$ has a smooth extension to the singular orbit if and only if it is smooth when restricted to every 2 plane in the slice $V$ containing $\dot c(0)$.
\end{theorem}
As we will see, it follows from the classification of transitive actions on spheres, that it is sufficient to require the condition only for a finite set of 2-planes $P_i=\{\dot c(0),v_i\}$, one for each irreducible summand in the isotropy representation of the sphere $K/H$. Thus  at most four 2-planes  are necessary.
 Furthermore,  $L_i=\exp(\theta v_i)\subset K$ is a closed one parameter group and hence the action of $L$ on $V$ and on a $K$ invariant complement  of $\mathfrak{k}$ in $\mathfrak{g}$ splits into 2 dimensional invariant subspaces $\ell_i$ isomorphic to $\mathbb{C}$, on which $L$ acts by multiplication with $e^{in_i\theta}$. The integers $n_i$ are determined by the weights of the representation of $K$ on $V$ and the tangent space of $G/K$. These integers will determine the smoothness conditions, see Table B and C.

To be more explicit,  choose a normal geodesic $c\colon[0,\infty)\to V$ orthogonal to all orbits. The metric on the regular part is determined  by its values along $c$, and via the action of $G$ this determines the metric on $M$.
Denote by $\mathfrak{g},  \mathfrak{h}$ the Lie algebras of $G$ and $ H$, and let $\mathfrak{n}$ be an $\Adj_H$ invariant complement of $\mathfrak{h}\subset\mathfrak{g}$. Since the stabilizer group along $c$  is constant equal to $H$, $\mathfrak{n}$ can be identified with the tangent space to the regular orbits along $c$ using action fields, i.e. $X\in\mathfrak{n} \to X^*(c(t))$.
Thus  $g=dt^2+h_t$, where $h_t, t>0$ is a family of $G$-invariant metrics  on the regular orbits $g\cdot c(t)=G/H$, depending  smoothly on $t$. Equivalently,  $h_t$ is a smooth family of $\Adj_H$ invariant inner products on $\mathfrak{n}$.

\smallskip

The metric is  described in terms of the length of Killing vector fields. We choose a basis $X_i$ of $\mathfrak{n}$ and let $X^*_i$ be the corresponding Killing  vector fields. Then $X^*_i(c(t))$ is a basis of $\dot c^\perp(t)\subset T_{c(t)}M$ for all $t>0$ and the metric is determined by the $r$ functions $g_{ij}(t)=g(X_i^*,X_j^*)_{c(t)},\ i\le j$.

Combining the finite set of smoothness conditions obtained from Theorem A, we will show that:

\begin{theorem}\label{smooth} Let $g_{ij}(t),\ t>0$ be a smooth family of positive definite matrices describing the cohomogeneity one metric on the regular part along a normal geodesic $c(t)$.     Then there exist integers $a_{ij}^k$ and $ d_k$, with $\ d_k\ge 0$,  such that the metric has a smooth extension to all of $M$ if and only if
$$\sum_{i,j} a_{ij}^k\,g_{ij}(t)=t^{d_k}\phi_k(t^2)\quad \text{ for } k=1,\cdots, r, \ \text{ and } t>0 $$
 where $\phi_1,\cdots,\phi_{r}$ are smooth functions defined for $t\ge 0$.
\end{theorem}

We will show that this system of $r$  equations can also be solved for the coefficients $g_{ij}$ of the metric. The integers  $a_{ij}^k $  are determined by  the Lie brackets $[X_i,X_j]$, and $ d_k$ by the integers $n_i$. These equations hold for all $t$ in the case of a complete metric on a non-compact manifold, and on the complement of the second singular orbit when the manifold is compact.
 We will illustrate in some specific examples  that it is straightforward to determine these integer.

\smallskip

The  problem of smoothness was studied in \cite{5} as well. There it was shown that smoothness is equivalent to showing that the k-th order Taylor polynomial of $g_{ij}$ is the restriction of an $\Adj_H$ invariant homogeneous polynomial of degree $k$ in  $\dim V$ variables with values in $S^2\mathfrak{n}$. In practice this description is difficult to apply, since one needs explicit expressions for these polynomials.

\smallskip

 In two future papers, we will  show that our new description is useful in proving general theorems about cohomogeneity one manifolds.
 In \cite{16} we classify curvature homogeneous cohomogeneity one metrics in dimension $4$, where the smoothness conditions at the singular orbit make the problem algebraically tractable. In \cite{17} we solve the initial value problem, starting at the singular orbit, for Einstein metrics, soliton metrics or for prescribing the Ricci tensor. The equations can be described in terms of the smooth functions $\phi_i$, and the system is smooth if and only if the values $\phi_i(0)$ satisfy certain compatibility conditions. These can be solved for some of the values $\phi_i(0)$, and the remaining ones are free parameters.  For this it is also important to understand the smoothness conditions for a symmetric $2$ tensor (in particular the Ricci tensor), which we indicate in Section 3.4.   The initial value problem for Einstein metrics was solved in \cite{5}, only under strong assumptions on the adjoint representation of $H$ on $\mathfrak{n}$ using different more complicated methods.

\smallskip

The paper is organized as follows. After discussing some preliminaries in Section 1, we prove Theorem A in Section 2. In Section 3 we describe how the action of the one parameter group $L\subset K$ on $V$ and on the tangent space to the singular orbit is used to derive  the smoothness conditions. This is an over determined system of equations, and we will show how it can be reduced to the system in Theorem B. In Section 4 we illustrate the method in some specific examples. There the reader will also find  step by step instructions of how the process works. In order to facilitate the procedure we determine the integers $d_k$ for the action of $K$ on $V$ in Section 5.

%%%%%%%%%%%%%%%%%%%%%%%%%%%%%%
%%%%%%%%%%%%%%%%%%%%%%%%%%%%%%%
\section{Preliminaries}%%%%%%%%%%%
%%%%%%%%%%%%%%%%%%%%%%%%%%%%%%
%%%%%%%%%%%%%%%%%%%%%%%%%%%%%%%

\bigskip

For a general reference for this Section see, e.g.,  \cite{1,2}.
A noncompact cohomogeneity one manifold is given by a homogeneous vector bundle
and a  compact one by  the union of two
homogeneous disc bundles. Since we are only interested in the smoothness conditions near a singular orbit, we restrict ourselves to only one such bundle. Let    $H,\, K ,\, G
$ be Lie groups with inclusions $H\subset K \subset G$ such that $H,K$ are compact and
$K/H=\Sph^{\ell}$. The transitive action of $K$ on
  $\Sph^{\ell}$ extends (up to conjugacy) to a unique linear action on the disc $V=\mathbb{R}^{{\ell}+1} $.
We can thus define the homogeneous vector bundle
$M=G\times_{K}V$ and $G$  acts on $M$  via left action in the first component. This action has principal isotropy group $H$, and singular isotropy group $K$ at a fixed base point $p_0\in G/K$ contained in the singular orbit.
  A disc $\Disc\subset V$ can be viewed as the slice of the $G$ action since, via the exponential map,  it can be identified $G$ equivariantly with a submanifold of M orthogonal  to the singular orbit at $p_0$.

\smallskip
Given a $G$-invariant metric $g$ on the regular part of the $G$ action, i.e. on the complement of $G\cdot p_0$, we want to determine when the metric can be extended smoothly to the singular orbit.
We choose   a geodesic $c$ parameterized by arc length and  normal to all orbits with $c(0)=p_0$. Thus, with the above identification, $c(t)\subset V$.
 At the regular points $c(t)$, i.e., $t>0$,
 the isotropy is constant  equal to  $H$.
We fix an $\Adj_H$ invariant splitting $\mathfrak{g}=\mathfrak{h}\oplus\mathfrak{n}$ and identify the tangent space $T_{c(t)} G/H=\dot{c}^\perp\subset T_{c(t)}M$,  with $\mathfrak{n}$
 via action fields: $X\in\mathfrak{n}\to
X^*(c(t))$. $H$ acts on $\mathfrak{n}$ via the adjoint representation
and a $G$ invariant metric on $G/H$ is described by an $\Adj_H$
invariant inner product on $\mathfrak{n}$. For $t>0$ the metric along $c$
is thus given  by $g=dt^2+h_t$ with $h_t$ a one parameter family of $\Adj_H$ invariant inner products on the vector space $\mathfrak{n}$, depending smoothly on $t$. Conversely, given such a family of inner products $h_t$, we define the metric  on the regular part of $M$ by using the action of $G$.

By the slice theorem, for the metric on $M$ to be smooth, it is sufficient that the  restriction to the slice $V$ is smooth. This restriction can be regarded as a map $g(t)\colon V\to S^2(\mathfrak{n})$.
The metric is defined and smooth on $V\setminus\{0\}$, and  we need to determine when it admits a smooth extension  to $V$.

We choose an $\Adj_H$ invariant  splitting
$$
\mathfrak{n}=\mathfrak{n}_0\oplus\mathfrak{n}_1\oplus\ldots\oplus\mathfrak{n}_r.
$$
where $\Adj_H$ acts trivially on $\mathfrak{n}_0$ and irreducibly on $\mathfrak{n}_i$ for $ i>0$. On  $\mathfrak{n}_i, i>0$ the inner product $h_t$ is  uniquely determined up to a multiple, whereas on $\mathfrak{n}_0$ it is arbitrary. Furthermore, $\mathfrak{n}_i$ and $\mathfrak{n}_j$ are orthogonal if the representations of $\Adj_H$ are inequivalent. If they are equivalent, inner products are described by $1,2$ or $4$ functions, depending on whether the equivalent representations are real, complex or quaternionic.

Next, we  choose a  basis $X_i$ of $\mathfrak{n}$, adapted to the above decomposition, and thus the metrics $h_t$ are described by a collection of smooth functions $g_{ij}(t)=g(X_i^*(c(t)),X_j^*(c(t)))$, $\ t> 0$.  In order to be able to extend this metric smoothly to the singular orbit, they must  satisfy certain  smoothness conditions at $t=0$, which we will discuss in the next two Sections. Notice that in order for the metric to be well defined on M, the limit of $h_t$, as $t\to 0$, must exist and be  $\Adj_K$ invariant  at the singular orbit.

 \smallskip

Choosing an $\Adj_K$ invariant complement to $\mathfrak{k}\subset\mathfrak{g}$, we obtain the decompositions
 $$\mathfrak{g}=\mathfrak{k}\oplus \mathfrak{m}  , \quad  \mathfrak{k}=\mathfrak{h}\oplus \mathfrak{p} \ \text{ and thus }\ \mathfrak{n}=\mathfrak{p} \oplus \mathfrak{m} . $$
 where we can also assume that $\mathfrak{n}_i\subset\mathfrak{p}$ or $\mathfrak{n}_i\subset\mathfrak{m}$.   Here $\mathfrak{m}$ can be viewed as the tangent space to the singular orbit $G/K$ at
  $p_0=c(0)$ and $\mathfrak{p}$ as the tangent space of the sphere $K/H\subset V$.

   It is important for us to identify $V$ in terms of action fields. For this we send  $X\in\mathfrak{p}$ to
  $\bar X:=\lim_{t\to 0}\frac{X^*(c(t))}{t}\in V$. Since $K$ preserves the slice $V$ and acts linearly on it, we thus have
  $X^*(c(t))=t\bar X\in V$.   In this language,    $V\simeq \dot c(0)\oplus \mathfrak{p}$. For simplicity we denote $\bar X$ again by $X$ and, depending on the context, use the same letter if considered as an element of $\mathfrak{p}$ or of $V$.

 Notice that since $K$ acts irreducibly on $V$,  an invariant inner product on $V$ is determined uniquely  up to a multiple. Since for any $G$ invariant metric we fix a geodesic $c$, which we assume is parameterized by arc length, this determines the inner product on $V$, which we denote by $g_0$.  Thus $g_0={g_{c(0)}}_{|V}$ for any $G$ invariant metric for which $c$ is a normal geodesic.

 $K$ acts via the
isotropy action $\Adj(K)_{| \mathfrak{m}}$ of $G/K$ on $\mathfrak{m}$ and via the slice
representation on $V$. The action on $V$ is determined by the fact that $K/H=\Sph^\ell$. Notice though that the action of $K$ on $\Sph^\ell$, and hence on $V$ is often highly ineffective. If $R\subset K$ is the ineffective kernel of the action, then there exists a normal subgroup $N\subset K$ with $K=(R\times N)/\Gamma$ where $\Gamma$ is a finite subgroup of the center of $R\times N$.  Thus $N$ acts almost effectively and transitively on $\Sph^\ell$ with stabilizer group $N\cap H$. We  list the almost effective actions by connected Lie groups acting transitively on spheres in Table A. From this, one can recover the action of $K$ on $V$ simply from the embedding $H\subset K$.

 The smoothness conditions only depend on the $\Ide$ component of $K$ since, as we will see, they are determined by certain one parameter groups $L\simeq S^1\subset K_0$.  Since also $L\subset N$, the smoothness conditions only depend on the $\Ide$ component of $N$ as well.

 We finally collect some specific properties of transitive actions on spheres.

 \begin{lemma}\label{spheres} Let $\Sph^\ell=K/H\subset V$ be a sphere, with $K$ acting almost effectively and  $H$  the stabilizer group of $v_0\in V$. If $\mathfrak{k}=\mathfrak{h}\oplus\mathfrak{p}$ is an  $\Adj_H$ invariant decomposition, we have:
 \begin{itemize}
 	\item[(a)] If $\mathfrak{p}_1\subset\mathfrak{p}$ is an  $\Adj_H$ irreducible summand with  $\dim\mathfrak{p}_1>1$, then $H$ acts transitively on the unit sphere in $\mathfrak{p}_1$,
 	\item[(b)]   If $\mathfrak{p}_i\subset\mathfrak{p}$, $i=1,2$, are two  $\Adj_H$ irreducible summands with  $\dim\mathfrak{p}_i>1$  and $X_1,Y_1\in \mathfrak{p}_1$ and $X_2,Y_2\in \mathfrak{p}_2$ two pairs of unit vectors, then there exists an $h\in H$ such that $\Adj(h)X_i=Y_i$.
 	\item[(c)] If $X\in \mathfrak{p}$ lies in an $\Adj_H$ irreducible summand, or a trivial one, then $\exp(tX)$ is a closed one parameter group in $K$ and leaves invariant the $2$-plane spanned by $v_0$ and $X^*(v_0)$.
 \end{itemize}	
 	\end{lemma}
 \begin{proof}
 Part (a) can be verified for each sphere separately, using the description of the adjoint representation, see e.g. \cite{18}.

 Part (b) is easily verified in case 5, 5' and 6,6' in \tableref{transitive}. In the remaining case of  $K=\Spin(9)$ and $H=\Spin(7)$ we have $\mathfrak{p}=\mathfrak{p}_1\oplus\mathfrak{p}_2$ with
$\Spin(7)$ acting on $\mathfrak{p}_1\simeq\mathbb{R}^7$  via the 2-fold cover $\Spin(7)\to\SOG(7)$, and  on $\mathfrak{p}_2\simeq\mathbb{R}^8$ via its spin representation. We can first choose an $h\in H$ with $\Adj(h)(X_1)=Y_1$. The claim then follows since the stabilizer of $H$ at $Y_1\in \mathbb{R}^7$ is $\Spin(6)$, and the restriction of the spin representation of $\Spin(7)$ on $\mathbb{R}^8$ to this stabilizer is the action of $\Spin(6)=\SUG(4)$ on $\mathbb{C}^4$, which is  transitive on the unit sphere.

Since $\exp(tX)$ is the flow of the action field $X^*$, part (c) is equivalent to saying that $\exp(tX)\cdot v_0$ is a great circle in $\Sph^\ell$. Recall that for a normal homogeneous metric, i.e. a metric on $K/H$ induced by a biinvariant metric on $K$, the geodesics are of the form $\exp(tX)\cdot v_0$ for some $X\in\mathfrak{p}$. This implies the claim if  $\Adj_H$ acts irreducibly on $\mathfrak{p}$.
  In all other cases, one can view the irreducible summand as the vertical or horizontal space of a Hopf fibration. The round metric on $\Sph^\ell$ is obtained from the metric induced by a biinvariant metric on $K$ by scaling the fiber, see
 \cite{10}, Lemma 2.4. But such a change does not change the geodesics whose initial vector is vertical or horizontal. By part (a), the one parameter groups  $\exp(tX)$ are either all closed in $K$, or none of them are. But for each transitive sphere one easily finds one vector $v$ where it is closed, see Section 6.
 \end{proof}

\bigskip

%%%%%%%%%%%%%%%%%%%%%%%%%%%%%%
%%%%%%%%%%%%%%%%%%%%%%%%%%%%%%%
\section{Reduction to a 2-plane}%%%%%%%%%%%
%%%%%%%%%%%%%%%%%%%%%%%%%%%%%%
%%%%%%%%%%%%%%%%%%%%%%%%%%%%%%%

\bigskip

In this Section we show how to reduce the question of smoothness of the metric on $M=G\times_KV$ to a simpler one.
If $\dim V=1$, i.e. the orbit $G/K$ is exceptional, smoothness (of order $C^k$ or $C^\infty)$ of the metric is equivalent to the invariance with respect to the Weyl group  since the slice is the normal geodesic. Recall that the Weyl group element is an element $w\in K$ such that $w(\dot c(0)=-\dot c(0)$, and is hence uniquely determined mod $H$. Hence we only need to discuss the conditions  at  singular points. i.e. $\dim V>1$.

\smallskip

 At a singular point,  the slice theorem for the action of $G$ implies that the metric is smooth if and only if its restriction to a slice $V$, i.e. $g_{|V}\colon V\to S^2(\mathfrak{p}\oplus\mathfrak{m})$ is smooth. Indeed, in a neighborhood $W$ of the slice we have an equivariant diffeomorphism $U\times V\to W: (x,p)\to \exp(x)p$, where $U$ is a sufficiently small neighborhood of $0\in\mathfrak{n}$.   We choose for each $\Adj_H$ irreducible summand in $\mathfrak{p}$ an (arbitrary) vector $v_i\ne 0$. If there exists a 3-dimensional trivial module  $\mathfrak{p}_0\subset\mathfrak{p}$,  we pick in $\mathfrak{p}_0$ an arbitrary fixed basis.
\bigskip

\begin{proposition}\label{2plane}
 A cohomogeneity one metric $g$ defined and smooth on the set of regular points in $M$  extends smoothly to the singular orbit  if and only if it is smooth when restricted to the 2 planes  $P_i\subset V$ spanned by $\dot c(0)$ and $v_i$.
\end{proposition}
\begin{proof} First notice that by \lref{spheres} (a), and since the metric is fixed along the normal geodesic $c$, the assumption implies that the metric is smooth when restricted to a 2 plane spanned by $\dot c(0)$ and $v$, where  $v$ is any  vector in an irreducible $\mathfrak{p}$ module.
	
	 It is sufficient  to show that $g(X,Y)_{|V}$ is smooth for any non-vanishing smooth vector fields $X,Y$ defined on $V$, i.e. $X,Y\colon V\to TM$.
	 We will use equivariance of the metric with respect to the action of $K$ on $V$. i.e.
 $$
 g(X,Y)(p)=g(k_*X,k_*Y)(kp) \text{ for all } p\in V\setminus\{0\}
 $$
 for the metric $g$ as well as all of its derivatives.

 We first define the metric at $0\in V$ and show it is $K$ invariant, as required. For this, define $g(X,Y)(0)=\lim_{t\to 0}g(X,Y)(c(t))$. If  $P_i$ is spanned by $\dot c(0)$ and $v_i$, then by \lref{spheres} (c) the one parameter group $L=\exp(tv_i)$ preserves the plane $P_i$ and equivariance with respect to $L\subset K$ implies that
 $g(0)$ is invariant under $L$.  By \lref{spheres} (a), the same is true for $\exp(tv)$ for any vector $v$ lying in an $\Adj_H$ irreducible submodule of $\mathfrak{p}$. But such one parameter groups, together with $H$, generate all of $K$. Indeed, this follows from the fact that $\frac{d}{dt}_{|t=0}\left(\exp(tv)\exp(tw)\right)=[v,w]$ and that $\mathfrak{h}\oplus[\mathfrak{p},\mathfrak{p}]$ is an ideal in $\mathfrak{g}$.

We next prove continuity. Let $p_i$ be a sequence of points $p_i\in V\setminus\{0\}$ such that $p_i\to 0$. We want to show that
$g(X,Y)(p_i)$  converges to $g(X,Y)(0)$. For this, let $w_0$ be an accumulation point of $w_i=p_i/|p_i|$ and choose a subsequence $w_i\to w_0$. Since $K$ acts transitively on a sphere in $V$, we can then choose $r_i\in K$ such that $r_iw_i=w_0$ and $r_i\to e\in K$, as well as $k_0\in K$ with $k_0w_0=\dot c(0)$. Setting $k_i=k_0r_i$, it follows that $k_iw_i=\dot c(0)$ with $k_i\to k_0$, which implies that $k_ip_i$ lies on the geodesic $c$.
Hence  equivariance of the metric, and continuity of the metric along the normal geodesic, implies that
$$g(X,Y)(p_i)=g({k_i}_* X,{k_i}_* Y)(k_i\cdot p_i)\to g({k_0}_* X,{k_0}_* Y)(0)=g( X, Y)(0) $$
where we also used that the metric at the origin is invariant under $K$. Since the same argument holds for any accumulation point of the sequence $w_i$, this proves continuity.

\smallskip

Next, we prove the metric is $C^1$. For simplicity we first assume that the action of $H$ on $\mathfrak{p}$   is irreducible and non-trivial and hence  $H$ acts transitively on the unit sphere in $\mathfrak{p}$. By assumption, the metric is smooth when restricted to the 2-plane   $P$  spanned by $v\in \mathfrak{p}$ and $\dot c(0)$.  Given a vector $w\in V$, possibly $w=\dot c(0)$,  we need to show that the derivative with respect to $w$ extends continuously across the origin, i.e. that
 \begin{equation}\label{con}
 \lim_{i\to\infty}\frac \partial{\partial w} g(X,Y)(p_i)=\frac \partial{\partial w} g(X,Y)(0)
 \end{equation}
   for  any sequence $p_i\in V$ with $p_i\to 0$. Let us first show that the right hand side derivative in fact exists. For this, since $K$ acts transitively on every sphere in $V$, we can choose $k\in K$ such that $kw\in P$ and hence:
\begin{eqnarray*}
   \frac \partial{\partial w} g(X,Y)(0)&=&\lim_{h\to 0}\frac {g( X, Y)(h \cdot w)-g( X,Y)(0)}h=\\
&=&\lim_{h\to 0}\frac {g( k_*X, k_*Y)(h \cdot kw)-g( k_*X,k_*Y)(0)}h
\end{eqnarray*}
   where we have used $K$ equivariance away from the origin and $K$ invariance of $g$ at the origin. But the right side is the derivative
   $$
    \frac \partial{\partial (kw)} g(k_*X,k_*Y)(0)
   $$
   which exists by assumption since $kw\in P$.

   Now choose as before $k_i\in K$ such that $k_ip_i$ lies on the geodesic $c$. Since
   $H$ acts transitively on the unit sphere in $\mathfrak{p}$, and since $\mathfrak{p}$ is the orthogonal complement to $\dot c(0)\in V$, we can choose $h_i\in H$ such that $h_ik_iw$ lies in $P$. As before, we can assume that $k_i\to k_0$ and $h_i\to h_0$. Equivariance and smoothness of the metric away from the origin implies that for each fixed $i$
\begin{eqnarray*}
\frac \partial{\partial w} g(X,Y)(p_i)=  \frac \partial{\partial (h_ik_iw)} g({(h_i k_i)}_* X,{(h_i k_i)}_* Y)( h_ik_i p_i)
\end{eqnarray*}
 Since $h_ik_i p_i=k_i p_i$ lies on the geodesic, and since $h_ik_iw\in P$, we get
\begin{eqnarray*}
&&\lim_{i\to\infty} \frac \partial{\partial w} g(X,Y)(p_i)= \frac \partial{\partial (h_0k_0w)}\ g({(h_0 k_0)}_* X,{(h_0 k_0)}_* Y)( 0)\\
&=&\lim_{h\to 0}\frac {g({(h_0 k_0)}_* X,{(h_0 k_0)}_* Y)(h \cdot h_0k_0w)-g({(h_0 k_0)}_* X,{(h_0 k_0)}_* Y)(0)}h\\
&=&\lim_{h\to 0}\frac {g( X, Y)(h \cdot w)-g( X,Y)(0)}h=\frac \partial{\partial w} g(X,Y)(0)
\end{eqnarray*}
Thus the metric is $C^1$. The proof proceeds by induction. Assume the metric is $C^{k}$.
 This means that $T(w_1,\dots,w_k,X,Y)(p)=\frac{ \partial^k}{\partial w_1\dots \partial w_k} g(X,Y)(p)$ is a smooth multi linear form on the slice $V$ which is  equivariant in all its  arguments. We can thus use the same proof as above  to show  that
   $$
  \frac\partial{\partial w}\left( \frac{ \partial^k}{\partial w_1\dots \partial w_k} g(X,Y)\right)(p)
   $$ extends continuously across the origin, and hence the metric is $C^{k+1}$.

   We now   extend the above argument to the case where $\mathfrak{p}$ is not irreducible.
    Let $P_i$ be the 2-plane spanned by $v_i$ and $\dot c(0)$.
 We first observe  that any vector in $\mathfrak{p}$, can be transformed by the action of $H$ into a linear combination of the vectors $v_i$. Indeed, if we look at the possible isotropy actions of $K/H$ in Table A, one sees that besides the trivial module (in which we chose a basis) there are at most two non-trivial modules and \lref{spheres}(b) implies the claim. Following the strategy in the previous case, we choose $k_i\in K$ such that $k_ip_i$ lies on the geodesic $c$, and $h_i\in H$ such that $h_ik_iw=\sum a_{ij}v_j$. Furthermore, $k_i\to k_0$ and $h_i\to h_0$ with $h_0k_0w=\sum a_{0j}v_j$.  By  linearity of the derivative, and since the metric is smooth on $P_i$ by assumption, we have
   \begin{eqnarray*}
&&\lim_{i\to\infty}\frac \partial{\partial (h_ik_iw)} g({(h_i k_i)}_* X,{(h_i k_i)}_* Y)( h_ik_i p_i)\\
&=& \lim_{i\to\infty}\sum_j a_{ij}\frac \partial{\partial v_j} g({(h_i k_i)}_* X,{(h_i k_i)}_* Y)( h_ik_i p_i)\\
&=&\sum_j \lim_{i\to\infty}a_{ij}\frac \partial{\partial v_j} g({(h_i k_i)}_* X,{(h_i k_i)}_* Y)( h_ik_i p_i)\\
&=&\sum_j a_{0j}\frac \partial{\partial v_j}\ g({(h_0 k_0)}_* X,{(h_0 k_0)}_* Y)( 0)\\
&=&\frac \partial{\partial (h_0k_0w)}\ g({(h_0 k_0)}_* X,{(h_0 k_0)}_* Y)( 0).
\end{eqnarray*}
    The proof now continues as before.
\end{proof}

\begin{remark}
Notice that unless the group $K$ is $\SpG(n)$ or $\SpG(n)\cdot \ULG(1)$, only one or two 2-planes are required. For the exceptions one needs four resp three  2-planes. Notice also, that we can choose any vector $v$ in an irreducible submodule in $\mathfrak{p}$. Indeed, the condition is clearly independent of such a choice since $H$ acts transitively on the unit sphere in every irreducible submodule.

\smallskip

We point out that    \pref{2plane} also holds for any tensor on $M$ invariant under the action of $G$, using the same strategy of proof.
\end{remark}

\bigskip

%%%%%%%%%%%%%%%%%%%%%%%%%%%%%%
%%%%%%%%%%%%%%%%%%%%%%%%%%%%%%%
\section{Smoothness on 2-planes}%%%%%%%%%%%
%%%%%%%%%%%%%%%%%%%%%%%%%%%%%%
%%%%%%%%%%%%%%%%%%%%%%%%%%%%%%%

\bigskip

In this Section we show that smoothness on 2-planes can be determined explicitly in a simple fashion.

 Recall that on $V$ we have the inner product  $g_0$ with $g_0={g_{c(0)}}_{|V}$ for any $G$ invariant metric with normal geodesic $c$. We fix a basis $e_0,e_1,\ldots,e_k$ of $V$, orthonormal in $g_0$, such that $c$   is given by the line  $c(t)=te_0=(t,0,\ldots,0)$. The tangent space to $M$ at the points $c(t), t>0$ can be identified with $\dot c(t)\oplus \mathfrak{m}\oplus\mathfrak{p}$ via action fields. The metric $g=dt^2+h_t$ on the set of regular points in $M$ is determined by a family of $\Adj_H$ invariant inner products $h_t$ on $ \mathfrak{m}\oplus\mathfrak{p} $, $t>0$, which depend smoothly on $t$. Furthermore, $\mathfrak{m}$ and $\mathfrak{p}$ are orthogonal at $t=0$, but not necessarily for $t>0$.
The inner products $h_t$  extend in a unique and smooth  way to  $ V\setminus\{0\}$ via the action of $K$. In order to prove smoothness at the origin, it is sufficient to show that $g(X_i,X_j)$ is smooth for some  smooth vector fields which are a basis at every point in a neighborhood of $c(0)$. For this we use the action fields $X_i^*$ corresponding to an appropriately chosen basis $X_i$ of $\mathfrak{m}$, restricted to the slice $V$, and   the (constant) vector fields $e_i$ on $V$. Recall also that
we identify $\mathfrak{p}$ with a subspace of $V$ by sending $X\in\mathfrak{p}$ to $\lim_{t\to 0}\frac{X^*(c(t))}{t}\in V$ and that  $X^*(c(t))=t X $.    Finally, we have the splitting $\mathfrak{p}=\mathfrak{p}_1\oplus\ldots\oplus\mathfrak{p}_s$ into $\Adj_H$ irreducible subspaces.

According to \pref{2plane}, it is sufficient to determine smoothness on a finite list of 2-planes. Let $P^*\subset V$ be one of those 2-planes, spanned by  $e_0=\dot c(0)$ and $X\in\mathfrak{p}_i$ for some $i$. We  normalize $X$ such that  $L:=\{\exp(\theta X)\mid \theta\in\mathbb{R},\ 0\le \theta\le 2\pi\}$ is a closed one parameter subgroup of $K$. By \lref{spheres},
the one parameter group   $L$ preserves $P^*$, but may not act effectively on it, even if $K$ acts effectively on $V$. Since $L\simeq S^1$, acting via rotation on $P^*$,  the   ineffective kernel is  $L\cap H$. Let $a$ be the order of the finite cyclic group $L\cap H$. Equivalently, $a$ is the largest integer with $\exp(\frac{2\pi}{ a}X)c(0)=c(0)$, or equivalently  $\exp(\frac{2\pi}{ a}X)\in H$. Thus $X/a$ has unit length in $g_0$ and $ L$ operates on $P^*$  as a rotation $R(a\theta)$ in the orthonormal basis $\dot c(0), X/a$.  We can also assume $a>0$ by replacing, if necessary, $ X $ by $- X$. This integer $a$ will be a crucial ingredient in the smoothness conditions. Notice that  $a$ is the same for any  vector $X\in\mathfrak{p}_i$ and we can thus simply denote it by $a_i$. In the Appendix we will compute the integers $a_i$ for each almost effective transitive action on a sphere.

The action of $L$ on $\mathfrak{m}$ decomposes $\mathfrak{m}$:
 $$\mathfrak{m}=\ell_0\oplus\ell_1\oplus\cdots\oplus\ell_r \text{ with } L_{|\ell_0}=\Ide, \text{ and } L_{|\ell_i}=R(d_i\theta)$$ for some integers $d_i$.
 Similarly we have a decomposition of $V$:
  $$V=\ell'_{-1}\oplus\ell_0'\oplus\ell_1'\oplus\cdots\oplus\ell_s'$$
with $\ell'_{-1}=\spam\{\dot c(0),X\},\    L_{|\ell_{-1}'}=R(a\theta),\ L_{|\ell'_0}=\Ide$  and $L_{|\ell_i '}=R(d_i'\theta).$
  We choose the basis $e_i$ of $V$ and $X_i$ of $\mathfrak{m}$ such that it is adapted to this decomposition  and oriented in such a way that  $a,d_i$ and $d_i'$ are positive. For simplicity, we denote the basis of $\ell_i$ by $Y_1,Y_2$, the basis of $\ell_i'$ by $Z_1, Z_2$, and reserve the letter $X$ for the one parameter group $L=\exp(\theta X)$. We choose the vectors $Z_i\in\mathfrak{p}$ such that they correspond to $e_{i+1}$ under the identification $\mathfrak{p}\subset V$ and hence $Z_i^*(c(t))=t e_{i+1}\in V $, as well as $X^*(c(t))=t e_0$.   We determine the smoothness of inner products module by module, and observe that an $L$ invariant function $f$ on $P^*$ extends smoothly to the origin if and only if its restriction to the line $te_0$ is even, i.e. $f(te_0)=g(t^2)$ with $g\colon (-\epsilon,\epsilon)\to\mathbb{R}$ smooth. Furthermore, we use the fact  that the metric $V\to S^2(\mathfrak{p}\oplus\mathfrak{m})$  is  equivariant with respect to the action of $K$, and hence $L$. Once the condition is determined when inner products are smooth when restricted to $P^*$, we restrict to the geodesic $c$ to obtain the smoothness condition for  $h_t$.

  \smallskip

 \noin   In the following, $\phi_i(t)$ stands for a generic smooth function defined on an interval $(-\epsilon,\epsilon)$.

\smallskip

We will separate the problem into three parts: smoothness of scalar products of elements in $\mathfrak{m}$, in $\mathfrak{p}$ and mixed scalar products between elements of $\mathfrak{m}$ and $\mathfrak{p}$. We will start with the easier case of the metric on $\mathfrak{p}$.

\subsection{Smoothness on $\mathfrak{p}$.}\label{slice}
  Recall that on a 2-plane a metric given in polar coordinates by $dt^2+f^2(t)\,d\theta^2$ is smooth if and only if $f$ extends to a smooth odd function with
 $f(0)=0$ and $f'(0)=1$, see, e.g., \cite{11}. If $X$ has unit length in the Euclidean metric $g_0$, we have $X^*=\frac{\partial}{\partial\theta}$ in the 2-plane spanned by $\dot c(0)$ and $X$. Hence smoothness on $\mathfrak{p}$ is equivalent to:

 \begin{equation}\label{ppart}
 g_{c(t)}(X^*,X^*)=  t^2 +t^4 \phi(t^2)\ \ \text{ for all } \ \  X\in\mathfrak{p}  \ \text{ with } \ g_0(X,X)=1
 \end{equation}

 \noindent for some smooth function $\phi$, defined on an interval $(-\epsilon,\epsilon)$.

 Notice that $\mathfrak{p}_i$ and $\mathfrak{p}_j$, for $i\ne j$, are orthogonal for any $G$ invariant metric, unless $(K,H)=(Sp(n),Sp(n-1))$, in which case there exists a 3 dimensional module $\mathfrak{p}_0$ on which $\Adj_H$ acts as $\Ide$. We choose three vectors $X_i\in\mathfrak{p}_0$, orthonormal in $g_0$.
Applying \eqref{ppart} to $(X_i^*+X_j^*)/\sqrt{2}$, it follows that   the metric is smooth on $\mathfrak{p}_0$ if and only if
\begin{equation}\label{ppart2}
g_{c(t)}( X_i^*,X_j^*)=t^2\delta_{ij}+t^4\phi_{ij}(t^2) \ \text{ where } \  X_i\in\mathfrak{p}_0,  \ K=Sp(n) \text{ and } \ g_0(X_i,X_j)=\delta_{ij}
\end{equation}
\noindent for some smooth functions $\phi_{ij}$.
\smallskip

It may sometimes be more convenient, as we do in the proofs, to normalize $X$ such that
   $L=\{\exp(\theta X)\mid  0\le \theta\le 2\pi\}$ is a closed one parameter group in $K$.  In that case, let $t_0$ be the first value such that $\exp(t_0 X)\in H$. Then $t_0=\frac{2\pi}{a}$ for $a=|L\cap H|$ and hence $X/a$ has unit length in $g_0$.

   Thus in this normalization we need to replace \eqref{ppart} by:

    \begin{equation}\label{ppart3}
   g_{c(t)}(X^*,X^*)= a_i^2 t^2 +t^4 \phi(t^2)\ \ \text{ for all } \ \  X\in\mathfrak{p}_i
   \end{equation}

 \noin  For  a 3-dimensional module $\mathfrak{p}_0$ we will see in Section 5 that $a_i=1$ and hence in this case \eqref{ppart2} remains valid.

 \smallskip

  See \cite{13} for a more detailed description.

\begin{remark}
 One easily modifies the smoothness conditions if the geodesic is not necessarily parameterized by arc length, but still orthogonal to the regular orbits. The only difference is that in this case $g_{c(t)}(\dot c,\dot c)= \psi(t) t^2$  and  $g_{c(t)}(X^*,X^*)=  \phi(t) t^2 $  for $X\in\mathfrak{p}$ with $\phi,\psi$ even and $\phi(0)=\psi(0)>0$, where $X$ has unit length in $g_0$. In the second normalization of $X$ we need that $\phi(0)=a^2\psi(0)$ if $g_{c(t)}(X^*,X^*)=  \phi(t) t^2 $ .
\end{remark}

\subsection{Inner products in $\mathfrak{m}$.}
\label{mproducts}
In the remaining sections $L= \{\exp(\theta X)\mid 0\le \theta \le 2\pi\}$ is a  one parameter group acting via $R(a\theta)$ on $\ell'_{-1}$.
 We first describe the inner products in a fixed  module $\ell_i$.

\begin{lemma}\label{irred} Let $\ell$  be an irreducible $L$ module in $\mathfrak{m}$ on which $L$ acts via a rotation $R(d\theta)$ in a basis $Y_1,Y_2$.
 If the metric on $\ell$ is given by  $g_{ij}=g_{c(t)}(Y^*_i,Y^*_j)$, then
$$
\left(
  \begin{array}{cc}
    g_{11} & g_{12} \\
    g_{12} & g_{22} \\
  \end{array}
\right)=   \left(
  \begin{array}{cc}
    \phi_1(t^2) & 0 \\
    0 & \phi_1(t^2) \\
  \end{array}
\right)+t^{\frac {2d}a}\left(
  \begin{array}{cc}
    \phi_2(t^2) & \phi_3(t^2) \\
    \phi_3(t^2) & -\phi_2(t^2) \\
  \end{array}
\right)
$$
for some smooth functions $\phi_k,\, k=1,2,3$.
\end{lemma}
\begin{proof}
The metric on $\ell$, restricted to the plane $ P^*\subset V$, can be represented by a matrix
$G(p)$ whose entries are functions of $p\in P^*$. We  identify $\ell\simeq\mathbb{C}$ and $P^*\simeq\mathbb{C}$ such that the action of $L$ is given by multiplication with  $e^{id \theta}$ on $\ell$ and $e^{ia \theta}$ on $P^*$. The  metric $G$ must be $L$ equivariant, i.e.
$$ G(p)=\begin{pmatrix} g_{11} &g_{12}\\ g_{12} &g_{22}\end{pmatrix} \ \text{ with } \ G(e^{ia \theta}p)=R(d\theta)G(p)R(-d\theta).$$

 The right hand side can also be  seen as a linear action of $L$ on $S^2\ell\simeq\mathbb{R}^3$ and we may describe it in terms of its (complex) eigenvalues and eigenvectors. We then get:
$$(g_{11}+g_{22})(e^{ia \theta}p)=(g_{11}+g_{22})(p)$$
$$(g_{12}+i (g_{11}-g_{22}))(e^{ia \theta}p)=e^{2d i\theta}(g_{12}+i (g_{11}-g_{22}))(p)$$
$$(g_{12}-i (g_{11}-g_{22}))(e^{ia \theta}p)=e^{-2d i\theta}(g_{12}-i (g_{11}-g_{22}))(p).$$
The first equality just reflects the fact that the trace is a similarity invariant. Let
$$w(p)=(g_{12}+i (g_{11}-g_{22}))(p).$$ Then the second equality says that
$w(e^{ia \theta} p)=e^{2 id \theta} w(p)$, and the third one is  the conjugate of the second. Setting $p=te_0, t\in \mathbb{R}$ and replacing $\theta$ by $\theta/a$, we get
$$
w(e^{i \theta} t)=e^{2 i\frac d a \theta} w(t)=(te^{i\theta})^{2 \frac da}\ t^{-2 \frac da} w(t).
$$
If we let $z=t e^{i\theta}$, then
\begin{equation*}
\label{expda}
w(z)=z^{2 \frac da}\frac{w(t)}{t^{2 \frac da} }\ \text{ or }\ z^{-2 \frac da} w(z)= t^{-2 \frac da} w(t), \ \text{ where }\  t=|z|.
\end{equation*}
The first equation says that if $w(z)$ is smooth, then $w(z)$ must have a zero of order $2\frac da$ at $z=0$. If so, the second equation says that
 the function $z^{-2 \frac da} w(z)$ is $L$-invariant.
 This means that $g_{11}+g_{22}$ and
$z^{-2 \frac da} w(z)$ must be  smooth functions of $|z|^2$. If we restrict $z^{-2 \frac da} w(z)$ to the real axis  and we  separate the real and the imaginary part this is equivalent to the existence of smooth functions $\phi_i$ such that
$$(g_{11}-g_{22})(t)=t^{2 \frac da} \phi_1(t^2),\qquad g_{12}(t)=t^{2 \frac da} \phi_2(t^2), \qquad (g_{11}+g_{22})(t)=\phi_3(t^2).$$
Conversely, given $3$ functions $g_{11},g_{22},g_{12}$ along the real axis that verify these relations, they admit a (unique) smooth $L$-invariant extension to $\mathbb{C}$. Indeed, the first two equalities guarantee that $z^{-2 \frac da} w(z)$ and hence $w(z)$ is a smooth function on $P^*$. The third equality guarantees that $g_{11}+g_{22}$, and hence $G(p)$, has a smooth extension to $P^*$.
\smallskip

\end{proof}

\begin{remark}\label{Weyl}
If $a$ does not divide $2d$,  the proof shows that $w(z)$ is smooth only if  $w(t)=0$ for all $t$. But then $g_{12}=0$ and $g_{11}=g_{22}$ is an even function. Thus in this Lemma, as well as in all following Lemmas's, in case of a fractional exponent of $t$, the term should be set to be $0$. In practice, this will follow already from $\Adj_H$ invariance.

Notice also that a Weyl group element is give by $w=\exp(i\frac{d}{a}\pi)$. Thus if $q=\frac{2d}{a}$ is odd, $w$ rotates the 2-plane $\ell$ and hence this module is not changed when it is necessary to select another one parameter group $L$.
\end{remark}

For inner products between different modules we have:

\begin{lemma}\label{irredproducts} Let $\ell_1$ and $\ell_2$ be two irreducible $L$ modules in $\mathfrak{m}$ with basis $Y_1,Y_2$ resp. $Z_1,Z_2$  on which $L$ acts via a rotation $R(d_i\theta)$with $d_i>0$.
If the inner products between $\ell_1$ and $\ell_2$ are given by $h_{ij}=g_{c(t)}(Y^*_i,Z^*_j)$, then
$$
\left(
  \begin{array}{cc}
    h_{11} & h_{12} \\
    h_{21} & h_{22} \\
  \end{array}
\right)= t^{\frac {|d_1-d_2|}a}  \left(
  \begin{array}{cc}
    \phi_1(t^2) & \phi_2(t^2) \\
    -\phi_2(t^2) & \phi_1(t^2) \\
  \end{array}
\right)+t^{\frac {|d_1+d_2|}a}\left(
  \begin{array}{cc}
    \phi_3(t^2) & \phi_4(t^2) \\
    \phi_4(t^2) & -\phi_3(t^2) \\
  \end{array}
\right)
$$
for some smooth functions $\phi_k$.
\end{lemma}
\begin{proof}
$L$ acts on $\ell_1\oplus\ell_2$ via conjugation with $\diag(R(d_1\theta),R(d_2\theta))$ and hence
$$
\left(
  \begin{array}{cc}
    h_{11} & h_{12} \\
    h_{21} & h_{22} \\
  \end{array}
\right)\to   R(d_1\theta)\left(
  \begin{array}{cc}
    h_{11} & h_{12} \\
    h_{21} & h_{22} \\
  \end{array}
\right)R(-d_2\theta)
$$
This action has eigenvectors
$$w_1=h_{11}+h_{22}+i (h_{12}-h_{21}),\qquad w_2=h_{12}+h_{21}-i (h_{11}-h_{22})$$
with eigenvalues $e^{(d_1-d_2) i \theta}$ and $e^{(d_1+d_2) i \theta}$,  and their conjugates.   We set
$$w_1(e^{a i\theta} p)=e^{|d_1-d_2| i \theta} w_1(p),\qquad w_2(e^{a i\theta} p)=e^{(d_1+d_2) i \theta} w_2(p),$$
where we replaced, if necessary, $w_1$ by its conjugate.
A computation similar to the previous ones shows that a smooth extension to the origin is equivalent to
$$(h_{11}+h_{22})(t)=t^{\frac {|d_1-d_2|}a} \phi_1(t^2),\qquad (h_{11}-h_{22})(t)=t^{\frac {d_1+d_2}a} \phi_2(t^2)$$
$$(h_{12}-h_{21})(t)=t^{\frac {|d_1-d_2|}a} \phi_3(t^2),\qquad (h_{12}+h_{21})(t)=t^{\frac {d_1+d_2}a} \phi_4(t^2)$$
where $\phi_i, i=1,\ldots,4$, are smooth real functions. Conversely, these relationships enable one to extend $h_{11}\pm h_{22}$ and $h_{12}\pm h_{21}$, and hence all inner products, smoothly to $P^*$.
\end{proof}

For inner products with elements in $\ell_0$ we have:

\begin{lemma}\label{irredtrivial} Let $\ell_0\subset\mathfrak{m}$ be the module on which $L$ acts as $\Ide$, and $\ell$ an irreducible $L$ module with basis $Y_1,Y_2$  on which $L$ acts via a rotation $R(d\theta)$.
\begin{itemize}
\item[(a)] If $Y\in \ell_0$,  then $g_{c(t)}(Y^*,Y^*)$ is an even functions of $t$,
\item[(b)] If $Y\in \ell_0$ and $h_i=g_{c(t)}(Y^*,Y_i^*)$, then
$$h_1(t)=t^{\frac da} \phi_1(t^2),\quad h_2(t)=t^{\frac da} \phi_2(t^2),$$
for some smooth functions $\phi_k$.
\end{itemize}
\end{lemma}
\begin{proof}
 If $Y\in \ell_0$,  then  $g(Y^*,Y^*)$ is invariant under $L$ and hence an even function.

In case (b), we consider the restriction of the metric to the three dimensional space spanned by $\ell$ and $Y$. This can be represented by a matrix $G(p)=\begin{pmatrix} g_{11} &g_{12} &h_1\\ g_{12} &g_{22} &h_2\\h_1 &h_2 &h\end{pmatrix}$  whose entries are functions of $p\in P^*$. In particular, $h_i=g(Y_i^*,Y^*)$.  The action of $L$ on  $G(p)$ is given by conjugation with $\diag(R(d\theta),1)$.  Decomposing into eigenvectors, we get, in addition to the eigenvectors already described in \lref{irred}, the eigenvector $w(z)=h_1(z)+i h_2(z)$ with eigenvalue $e^{di\theta}$. But
$w(e^{ia\theta}p)=e^{di\theta} w(p)$ implies that $z^{-\frac da} w(z)$ is an invariant function. Thus smoothness for the $h_i$ functions is  equivalent to
$$h_1(t)=t^{\frac da} \phi_1(t^2),\quad h_2(t)=t^{\frac da} \phi_2(t^2)$$
for some smooth functions $\phi_i$.
\end{proof}

\subsection{Inner products between $\mathfrak{p}$ and $\mathfrak{m}$}
\label{slicem}
Recall that for an appropriately chosen  basis $e_0,\ldots,e_k$ of $V$, we need to show that the inner products $g( e_i,X_j^*)$, where $X_i$ is a basis of $\mathfrak{m}$, are smooth functions when restricted to the plane $P^*\subset V$. When restricting to the geodesic $c$, we obtain the smoothness conditions on the corresponding entries in the metric.

Recall also that the plane $P^*$ is spanned by $e_0=\dot c$ and $X\in\mathfrak{p}\subset V$ such that
$L=\{\exp(\gth X)\mid \theta\in\mathbb{R}\}$ is a closed one parameter group in $K$. We also have the decomposition of $V$:
  $$V=\ell'_{-1}\oplus\ell_0'\oplus\ell_1'\oplus\cdots\oplus\ell_s'$$
with $\ell'_{-1}=\spam\{\dot c(0),X\},\    L_{|\ell_{-1}'}=R(a\theta),\ L_{|\ell'_0}=\Ide$ and $L_{|\ell_i'}=R(d_i'\theta)$
which we use in the following. Finally, recall that $Z^*(c(t))=tZ\in V$ for $Z\in\mathfrak{p}$ and that $g_{c(t)}(\frac{\partial}{dt},X^*)=0$ for all $X\in\mathfrak{p}\oplus\mathfrak{m}$.

\begin{lemma}\label{sliceX}
 Let $X\in \ell'_{-1}$.  Then we have:
\begin{itemize}
\item[(a)] If $Y\in \ell_0$,  then $g_{c(t)}(X^*,Y^*)=t^2\phi(t^2)$,
\item[(b)] If  $Y_1,Y_2$ a basis of the irreducible module $\ell=\ell_i$, on which $L$ acts as $R(d\theta)$ with $d>0$, then
$g_{c(t)}(X^*,Y_k^*)=t^{\frac {d}a+2}\phi_k(t^2)$
\end{itemize}
for some smooth functions $\phi, \phi_k$.
\end{lemma}
\begin{proof}
For part (a) the proof is similar to    \lref{irredtrivial}.
On the 3-space spanned by $e_0=\dot c(0),\ e_1=X,\ e_2=Y$, the one parameter group $L$ acts via conjugation with $\diag(R(a\theta),1)$ and, using that fact that $Y^*$ is orthogonal to $\dot c$, the metric is given by
$G(p)=\begin{pmatrix} 1 &0 &0\\ 0 &1 &h\\0 &h &f\end{pmatrix}$
with $h=g( e_1,Y^*)$ and $f=g(Y^*,Y^*)$. We already saw that $f$ is an even function, and as in the proof of \lref{irredtrivial}, we see, when restricted to the geodesic, $h(t)=t^{\frac aa}\phi(t^2)=t\phi(t^2)$. Hence
$g_{c(t)}(X^*,Y^*)=tg_{c(t)}(e_2,Y^*)=t^2\phi(t^2)$.

For part (b) the proof is similar to    \lref{irredproducts}.
On the $4$ dimensional space spanned by $e_0,e_1$ and $Y_1,Y_2$ the group $L$ acts via conjugation by $\diag(R(a\theta,R(d\theta))$ and the metric is given by
$$G(p)=\begin{pmatrix} 1 &0 &0&0\\ 0 &1 &h_1 &h_2\\0 &h_1&g_{11} &g_{12}\\
0&h_2&g_{12}&g_{22}\end{pmatrix}$$
with $h_k=g( e_1,Y_k^*)$ and $g_{kl}=g( Y_k^*,Y_l^*)$. As in the proof of \lref{irredproducts} it follows that
$$
h_{2}(t)=t^{\frac {|d-a|}a} \phi_1(t^2)\ \ \text{ and }\ \ h_{2}(t)=t^{\frac {|d+a|}a} \phi_2(t^2)
$$
and hence $h_{2}(t)=t^{\frac {d}a+1} \phi(t^2)$,  and similarly for $h_1(t)$. Thus
$g_{c(t)}(X^*,Y_k^*)=tg_{c(t)}(e_2,Y_k^*)=   t^{\frac {d}a+2}\phi_k(t^2)  $.
\end{proof}

Next the inner products with  $\ell_0'$.

\begin{lemma}\label{slicetrivial} For $Z\in \ell_0'$ we have:
\begin{itemize}
\item[(a)] If $Y\in \ell_0$,  then $g_{c(t)}( Z^*,Y^*)=t^3\phi(t^2)$,
\item[(b)] If  $Y_1,Y_2$ is a basis of the irreducible module $\ell_i$, then
$$g_{c(t)} (Z^*,Y_k^*)=t^{\frac {d_i}a+1}\phi_k(t^2)$$
\end{itemize}
for some smooth functions $\phi_i$.
\end{lemma}
\begin{proof}
For part (a), let $Z=e_1$. Then  $g( e_1,Y^*)$ is $L$ invariant and hence even. Furthermore, it vanishes at $t=0$ since the slice is orthogonal to the singular orbit at $c(0)$. Hence $g( e_1,Y^*)=t^2\phi(t^2)$, which implies $g_{c(t)}( Z^*,Y^*)=t\,t^2\phi(t^2)$.

 Similarly for (b), using the proof of \lref{irredtrivial}, it follows that $g_{c(t)} (e_1,Y_k^*)=t^{\frac {d_i}a}\phi_k(t^2)$. Since $d_i,a>0$, this already vanishes as required. The proof now finishes as before.
\end{proof}

 And finally the remaining inner products:

\begin{lemma}\label{sliceirred} Let $\ell_i'$ and $\ell_j$ with $i,j>0$ be two irreducible $L$ modules with basis $Z_1,Z_2$ resp. $Y_1,Y_2$  on which $L$ acts via a rotation $R(d_i'\theta)$ resp. $R(d_j\theta)$ with $d_i', d_j>0$.

 \begin{itemize}
\item[(a)] The inner products $h_{kl}=g_{c(t)}(Z^*_k,Y^*_l)$ satisfy
\begin{eqnarray*}
\left(
  \begin{array}{cc}
    h_{11} & h_{12} \\
    h_{21} & h_{22} \\
  \end{array}
\right)&=& t^{\frac {|d_i'-d_j|}a+b}  \left(
  \begin{array}{cc}
    \phi_1(t^2) & \phi_2(t^2) \\
    -\phi_2(t^2) & \phi_1(t^2) \\
  \end{array}
\right)+\\
&+&t^{\frac {|d_i'+d_j|}a+1}\left(
  \begin{array}{cc}
    \phi_3(t^2) & \phi_4(t^2) \\
    \phi_4(t^2) & -\phi_3(t^2) \\
  \end{array}
\right)
\end{eqnarray*}
where $b=3$ if $d_i'=d_j$, and $b=1$ if $d_i'\ne d_j$,
\item[(b)]If $Y\in \ell_0$, then
$g_{c(t)} (Y^*,Z_k^*)=t^{\frac {d_i'}a+1}\phi_k(t^2)$
\end{itemize}
for some smooth functions $\phi_i$.
\end{lemma}

\begin{proof}
 (a) We repeat the proof of \lref{irredproducts} for the basis $e_1=Z_1,e_2=Z_2,e_3=Y_1,e_4=Y_2$ of $\ell_i'\uplus\ell_j$. But if $d_i'=d_j$, we have to require in addition that the inner products vanish at $t=0$, i.e., $\phi_1(0)=\phi_2(0)=0$, which means the first matrix must be multiplied by $t^2$. The proof then proceeds as before.

 (b) We proceed as in \lref{irredtrivial} (b).
\end{proof}
This finishes the discussion of all possible inner products in $\mathfrak{n}=\mathfrak{p}\oplus\mathfrak{m}$.

\bigskip
\subsection{Smoothness conditions for symmetric $2\times 2$ tensors}
The above methods can be applied to obtain the smoothness conditions for any $G$ invariant tensor, defined along a curve $c$ transverse to all orbits. One needs to take care though, since for a metric $g$ the slice and singular orbit are orthogonal at $t=0$, whereas for a general tensor this may not be the case.   For the purpose of applying this to the Ricci tensor,  we briefly discuss how to derive the smoothness conditions for any symmetric $2\times 2$ tensor $T$.
\smallskip

 The proofs in Section 3.2 show that for the functions $T(\mathfrak{m},\mathfrak{m})$  the conditions for $T$ and a metric $g$ are the same.

  For $T(\mathfrak{p},\mathfrak{p})$ the only difference is that now $T(X_i^*,X_j^*)= \phi_0 t^2\delta_{ij} + \phi_{ij}(t^2)t^4$ for $X_i\in\mathfrak{p}$,  where $X_i$ has unit length in $g_0$ and $\phi_0$ is a real number, which is allowed to be $0$. Notice also that $T(\mathfrak{p}_i,\mathfrak{p}_j)=0$ for $0< i < j$ since the $\Adj_H$ representations are inequivalent.

 A new feature is that, unlike in the case of a metric, the mixed terms  $T(\dot c(t),X^*)$ do not have to vanish if $X\in\mathfrak{p}\oplus\mathfrak{m}$ lies in a module on which $\Adj_H$ acts trivially. For the case of $X\in \mathfrak{p}_0$ one easily sees that:
  \begin{equation}\label{po}
 T_{c(t)}( \dot c,\dot c)=\psi_1(t^2), T_{c(t)}( \dot c,X^*)=t\psi_2(t^2),  T_{c(t)}( X^*,X^*)=t^2\psi_3(t^2),
  \end{equation}
with $\psi_1(0)=\psi_3(0)=\phi_0$.

  For the  case of  $T(\dot c,\mathfrak{m}_0)$, as well as   $T(\mathfrak{p},\mathfrak{m})$, one needs to examine the proof of the Lemma's in Section 3.3, keeping in mind that the values of $T$ on the 2-plane $\ell_{-1}=\{ \dot c, X\}$ is now more generally given by \eqref{po}.  In some cases, for a metric tensor, certain components are forced to have a zero of two orders higher at $t=0$ than a generic symmetric tensor since the regular orbits are orthogonal to the geodesic $c$. One easily sees that the conditions in \lref{sliceX}(a),  \lref{slicetrivial}(b) and \lref{sliceirred}(b) are the same, whereas  in \lref{sliceX}(b), \lref{slicetrivial}(a) and in \lref{sliceirred}(a) when $d_i'=d_j$, the allowed order for $T$ is two less.  We summarize the results in Table D.   This difference is important when studying Einstein metrics, or prescribing the Ricci tensor, see \cite{17}.
\smallskip

%%%%%%%%%%%%%%%%%%%%%%%%%%%%%%
%%%%%%%%%%%%%%%%%%%%%%%%%%%%%%%
\section{Examples}%%%%%%%%%%%
%%%%%%%%%%%%%%%%%%%%%%%%%%%%%%
%%%%%%%%%%%%%%%%%%%%%%%%%%%%%%%

\bigskip

Before we illustrate the method with some examples, let us make some general comments.
\smallskip

We can choose an inner product $Q$ on $\mathfrak{g}$ which is $\Adj_K$ invariant on $\mathfrak{m}$, equal to $g_0$ on $\mathfrak{p}$ under the inclusion $\mathfrak{p}\subset V$, and such that the decomposition $\mathfrak{n}=\mathfrak{p}\oplus\mathfrak{m}$ is orthogonal.

 If $G$ is compact, one often starts with a biinvariant metric $Q$ on $\mathfrak{g}$. We point out though, that then $Q_{|\mathfrak{p}}$ is not always a multiple of the metric $g_0$.  Thus one needs to determine the real numbers $r_i>0$ such that $Q_{|\mathfrak{p}_i}=r_i g_0$, $i=1,\cdots,s$, which needs to be used in order to translate   the conditions in \eqref{ppart} into a basis of $\mathfrak{p}$ orthonormal in $Q$. We point out that if $s>1$, $r_i$ depends on $i$ since in that case the biinvariant metric  $Q_{| K}$ does not restrict to a constant curvature metric on $K/H$.
See Table 2.5 in  \cite{10}  for the values of $r_i$.

\smallskip

Smoothness is determined by the one parameter groups  $L=\{\exp(\gth v)\mid 0\le \gth \le 2\pi\}$, one for each irreducible $\mathfrak{p}$ module.
Since the action of $L$ on $\mathfrak{m}$ is given by the restriction of $\Adj_K$, the exponents $d_i$ can be determined in terms of Lie brackets, i.e. on $\ell_i$ we have

\begin{equation}\label{di}
[v,Y_1]=d_iY_2\ \  \text{ and }\ \  [v,Y_2]=-d_iY_2
\end{equation}
 where $Y_1,Y_2\in \ell_i$ are $Q$ orthogonal vectors of the same length. This also determines the  orientation of the basis so that $d_i>0$.
 The decomposition under $L$ can be recovered from the weight space decomposition of the action of $K$ on $\mathfrak{m}$ with respect to a maximal abelian subalgebra containing $v$. Thus, on each irreducible $K$ module in $\mathfrak{m}$, we have  $d_i=\alpha_i(v)$, for all weights $\alpha_i$, and hence the largest integer is $\lambda(v)$ where $\lambda$ is the dominant weight.

The slopes $d_i'$ are not determined by Lie brackets. One needs to use the knowledge of the embedding $H\subset K$  to determine the action of $K$, and hence $L$, on $V$. For the almost effective actions of $K$ on spheres, a choice of the vectors $v$ and the values of $a$ and $d_i'$ will be described in Section 6.

The  functions $g_{ij}(t)$ determining the metric  are usually given in terms of a decomposition of $\mathfrak{h}^\perp=\mathfrak{n}$ into $\Adj_H$ irreducible modules. But the decomposition
of $\mathfrak{m}$ into  irreducible modules under $L_i$ are usually quite different. Thus the entries of the metric in the  Lemmas of Section 3 are  linear combinations of $g_{ij}$. Furthermore, for different 2-planes $P_i^*$, the decomposition under $L_i=\exp(\gth v_i)$ is again typically not the same since the vectors $v_i$ do not lie in a common maximal torus. One may thus obtain different smoothness conditions for different one parameter groups $L_i$ which need to be combined to obtain the full smoothness conditions.

One can now row reduce these equations, which gives rise to relationships between the even functions. Substituting these, one can then express the $k$ metric coefficients in terms of $k$ even functions.

The conditions of order $0$ are equivalent to $K$ invariance. The conditions of order $1$ are equivalent to equivariance of the second fundamental form $B\colon S^2T\to T^\perp=V$ of the singular orbit $G/K$ with tangent space $T=T_{p_0}K/H$ under the action of $K$. Recall also that one has a Weyl group element $w\in K$ with $w(\dot c(0))=- \dot c(0)$, uniquely determined$\mod H$. Clearly $w\in L_i$ for all $i$, in fact $w=\exp(\frac\pi{a}v_i)$ up to a change by an element in $\Adj_H$. The property of the length squared being even or odd functions is already determined by the action of the Weyl group element on $\mathfrak{m}$, see \rref{Weyl} and Section 5.

\bigskip

Summarizing the method one needs to use the following steps:

\begin{itemize}
	\item[(a)] Decompose $\mathfrak{n}$ into $\Adj_H$ irreducible modules, which determines the coefficients $g_{ij}$ of the metric,
	\item[(b)] Choose one parameter groups $L=\exp(tX)$, one for each irreducible $\mathfrak{p}$ module. See Section 6 for convenient choices, as well as the value of the integers $a$ and $d_i'$,
		\item[(c)] Decompose $\mathfrak{m}$ into the sum of 2-dimensional modules $\ell_i$ under the action of $L$ and determine the integers $d_i$, using e.g. the description \eqref{di},
			\item[(d)] Express the coefficients of the inner products in $\ell_i$ in terms of the metric coefficients $g_{ij}$,
				\item[(e)] Use Table A and B to express the smoothness conditions in terms of even functions $\phi_i$,
					\item[(f)] Row reduce the equations coming from all one parameter groups $L_i$, and replace some of the even functions in terms of others,
					\item[(g)] Solve the resulting system of equations for the metric $g_{ij}$.
\end{itemize}

\bigskip

 \begin{center} {\bf Example 1}\end{center}

 \bigskip

 A simple  example is given by the groups $G=Sp(1)\times S^1, K=\{(e^{j\theta},1)\mid 0\le \theta\le 2\pi\}\cdot H$ and $ H\simeq \mathbb{Z}_4$ with generator $(i,i)$.
  There exists an infinite family of inequivalent cohomogeneity one actions on $S^5$ as a special case of the Kervaire sphere examples, see \cite{9}, the simplest one being  the tensor product action of $SO(3)SO(2)$ on $S^5$. For all of them one half of the group diagram is given by the above groups. Notice that the action of $K$ on the slice $V\simeq \mathbb{C}$ is given by $(q,z)\cdot v = zv$.

  If we let $X_1=(i,0),\; X_2=(j,0),\; X_3=(k,0)$ and $Y=(0,i)$ then we have the $\Adj_H$ invariant decomposition
$\mathfrak{p}=\mathfrak{k}=\mathbb{R}\cdot X_2$  and $\mathfrak{m}=\mathfrak{m}_0\oplus\mathfrak{m}_1$ with $\mathfrak{m}_0=\spam\{X_1,Y\} $, $\mathfrak{m}_1=\mathbb{R}\cdot X_3$. Since $\Adj_H$ acts as $\Ide$ on $\mathfrak{m}_0$ and as $-\Ide$ on $\mathfrak{p}\oplus \mathfrak{m}_1$ the nonvanishing inner products are given by
$$
f_i=\langle X_i,X_i\rangle , i=1,2,3,\quad g=\langle Y,Y\rangle,\quad h_1=\langle X_1,Y\rangle,\quad h_2=\langle X_2,X_3\rangle.
$$
 There is only the one parameter group $L=\{\exp(\theta X_2)\mid 0\le \theta\le 2\pi\}$ to be considered.   $L$ acts via $R(\theta)$ on $\ell_{-1}'=\spam\{ \dot c(0), X_2\}$, trivially on $\ell_0=\mathbb{R}\cdot Y$, and by $R(2\theta)$ on $\ell_1=\spam\{X_1,X_3\}$. Thus $a=1$ and $d_1=2$. According to Table B and C we have
 
 $$
f_1=\phi_5(t^2)+t^4\phi_6(t^2),\quad f_3=\phi_5(t^2)-t^4\phi_6(t^2),\quad  \quad g= \phi_2(t^2),
 $$
 and
 $$
f_2=t^2+t^4\phi_1(t^2), \quad h_1= t^2\phi_3(t^2), \quad h_2=t^4\phi_4(t^2).
 $$

 \smallskip

See also \cite{8} Appendix 1 for a further class of examples with $K/H \simeq S^1$.

\bigskip

\begin{center} {\bf Example 2}\end{center}
\bigskip

In \cite{3}, the author studied cohomogeneity one Ricci flat metrics on the homogeneous disk bundle with $H=T^2\subset K=U(2)\subset G=SU(3)$, where we assume that $U(2)$ is the lower $2\times 2$ block. We illustrate that the smoothness conditions can be obtained with our methods quickly.

Let $E_{kl},iE_{kl}$, $k<l$, be the usual basis of $\mathfrak{su}(3)$. Then the decomposition of $\mathfrak{h}^\perp$ into  $\Adj_H$ irreducible representations is given by:
$$
 \mathfrak{n}_1=\{E_{23},iE_{23}\},\ \mathfrak{n}_2=\{E_{12},iE_{12}\},\ \mathfrak{n}_3=\{E_{13},iE_{13}\}
$$
Since they are all inequivalent, the metric is determined by:
$$
f_1=|E_{23}|^2=|iE_{23}|^2,\ f_2=|E_{12}|^2=|iE_{12}|^2,\ f_3=|E_{13}|^2=|iE_{13}|^2
$$
The module $\mathfrak{p}=\mathfrak{n}_1$ is irreducible since $K/H=\Sph^2$ and we can choose $L=\exp(\theta E_{23})$. Since $\exp(\pi E_{23} )\in H$, we have $a=2$ and hence $\frac12  E_{23}$ has unit length in the Euclidean inner product $g_0$ on the slice. The decomposition under $L$ is
$$
\mathfrak{m}=\ell_1\oplus\ell_2 \text{ with } \ell_1=\{E_{12},E_{13},\}, \ell_2=\{iE_{12},iE_{13},\}  \text{ and } d_1=d_2=1
$$
Since $\mathfrak{p}$ is orthogonal to $\mathfrak{m}$, the decomposition of the slice $V$ is not needed.
Thus the metric is smooth if and only if
$$
f_1=4t^2,\ f_2+f_3=\phi_1(t^2),\  f_2-f_3=t\phi_2(t^2)
$$
for some smooth functions $\phi_1,\phi_2$.

\bigskip

 \begin{center} {\bf Example 3}\end{center}

 \bigskip

Let $H\subset K\subset G$ be given by $\SOG(2)\subset \SOG(3)\subset\SOG(5)$, where the embedding of $\SOG(3)$ in $\SOG(5)$ is given by the unique irreducible representation of $\SOG(3)$ on $\mathbb{R}^5$. The singular orbit $G/K$ is the  Berger space (which is positively curved in a biinvariant metric).

We consider the following basis of $\mathfrak{g}=\mathfrak{so}(5)$:
$$
K_1=2 E_{12}+E_{34},\quad K_2=E_{23}-E_{14}+\sqrt{3} E_{45},\quad
K_3=E_{13}+E_{24}+\sqrt{3} E_{35},$$
$$V_1=\frac 1{\sqrt{5}} E_{12}-\frac 2{\sqrt{5}} E_{34},\quad
V_2=\frac {\sqrt{2}}{\sqrt{5}} E_{45}-\frac {\sqrt{3}}{\sqrt{10}} (E_{23}-E_{14}),$$
$$V_3=\frac {\sqrt{2}}{\sqrt{5}} E_{35}-\frac {\sqrt{3}}{\sqrt{10}} (E_{13}+E_{24}),\quad V_4=E_{25},$$
$$V_5=E_{15},\quad V_6=\frac 1{\sqrt{2}} (E_{24}-E_{13}),\quad V_7=-\frac 1{\sqrt{2}} (E_{23}+E_{14}).
$$

Then $K_1,K_2, K_3$ span the subalgebra $\mathfrak{k} \simeq \mathfrak{so}(3)$ with $[K_1,K_2]=K_3$ and cyclic permutations. Thus $K_i$ is orthonormal with respect to the biinvariant metric $Q_{\mathfrak{so}(3)}(A,B)=-\frac12 \trace(AB)$ which induces the metric of constant curvature $1$ on $SO(3)/SO(2)=\Sph^2$.  We choose the base point such that the Lie algebra of its stabilizer group $H$ is spanned by $K_1$. Hence $\dot c(0), K_2,K_3$ is  an orthonormal basis in the inner product $g_0$ on $V=\mathbb{R}^3$.  Notice that for the biinvariant metric $Q_{\mathfrak{so}(5)}(A,B)=-\frac12 \trace(AB)$ we have $Q_{\mathfrak{so}(5)}(A,B)=5Q_{\mathfrak{so}(3)}(A,B) $ for $A,B\in\mathfrak{so}(3)$.  Thus, if we abbreviate $Q=Q_{\mathfrak{so}(5)}$, we have $Q( K_i,K_j)= 5\delta_{ij}$. On the other hand,  $V_i$ are orthonormal unit vectors in $Q$.

We have the following decomposition of $\mathfrak{p}\oplus\mathfrak{m}$ as sum of irreducible $H$-modules:
$$\mathfrak{p}=\spam(K_2,K_3),\qquad \mathfrak{m}_0=\spam(V_1),\qquad  \mathfrak{m}_1=\spam(V_2,V_3),$$
$$\mathfrak{m}_2=\spam(V_4,V_5), \qquad \mathfrak{m}_3=\spam(V_6,V_7).$$
$\Adj_H$ acts trivially on $\mathfrak{m}_0$, with speed one on $\mathfrak{p}$ and $\mathfrak{m}_1$, and with speed 2 and 3 on $\mathfrak{m}_2$ resp. $\mathfrak{m}_3$. E.g., since $H=\{\exp(tK_1)\mid 0\le t\le 2\pi\}$, one needs to check that $[K_1,V_4]=2V_5$ and $ [K_1,V_5]=-2V_4$.
 Thus $\mathfrak{p}$ and $\mathfrak{m}_1$ are equivalent as $H$-modules while all the other modules are inequivalent. An $\Adj_H$ invariant metric $g$ along $c(t)$ is
 thus defined by the following functions:
$$f=\langle K_2,K_2\rangle =\langle K_3,K_3\rangle,\quad  g_1=\langle V_1,V_1 \rangle,\quad g_2=\langle V_2,V_2\rangle =\langle V_3,V_3\rangle,$$
$$g_3=\langle V_4,V_4\rangle =\langle V_5,V_5\rangle,\quad g_4=\langle V_6,V_6\rangle=\langle V_7,V_7\rangle.$$
$$h_{11}=\langle K_{2},V_{2}\rangle=\langle K_{3},V_{3}\rangle,\quad h_{12}=\langle K_{2},V_{3}\rangle=-\langle K_{3},V_{2}\rangle.$$
and all other scalar products are zero.

For the  smoothness conditions, since  $\Adj_H$ acts irreducibly on $\mathfrak{p}$, we need to choose only one vector and set $X=K_2$ with $L=\exp(tK_2)\subset SO(3)$.
 Since $SO(3)$ acts standard on $V$, we have $a=1$. Furthermore, $V=\ell_{-1}'\oplus\ell_0$ with $\ell_{-1}'=\spam\{\dot c(0),K_2\}$ and $\ell_0'=\spam\{K_3\}$ since $L$ acts via rotations in the $\dot c(0),K_2$ plane, and hence trivially on $e_3=K_3^*(0)$.

Under the action of $L$, one easily sees that  $\mathfrak{m}$  decomposes as the sum of the following
irreducible modules:
\begin{align*}
\mathfrak{l}_0&=\spam(\sqrt{6} V_2+\sqrt{10} V_7),& \mathfrak{l}_1&=\spam(V_3+\sqrt{15} V_6,-\sqrt{6} V_1-\sqrt{10} V_4)\\
\mathfrak{l}_2&=\spam(\sqrt{10} V_2-\sqrt{6} V_7, 4 V_5),& \mathfrak{l}_3&=\spam(-\sqrt{15} V_3+ V_6,\sqrt{10} V_1-\sqrt{6} V_4).
\end{align*}
and a Lie bracket computation shows that under the action of $L$ we have $d_i=i$ for $ i=1,2,3$. E.g. $[K_2, -\sqrt{15} V_3+ V_6]=3(\sqrt{10}V_1-\sqrt{6} V_4)$ and $[K_2, \sqrt{10} V_1-\sqrt{6} V_4]=-3(-\sqrt{15} V_3+ V_6)$.

\bigskip

{ 1) \it Irreducible modules in $\mathfrak{m}$.}
We have three irreducible $L$-modules in $\mathfrak{m}$ and for each of them we apply \lref{irred}, and  use the notation $g_{ij}$ therein. Notice that due to $\Adj_H$ invariance, all vectors $V_i$ are orthogonal to each other.

For $\ell_1$ we have:
$$g_{11}=\langle V_3+\sqrt{15} V_6,V_3+\sqrt{15} V_6\rangle=g_2+15 g_4$$
$$g_{22}=\langle-\sqrt{6} V_1-\sqrt{10} V_4,-\sqrt{6} V_1-\sqrt{10} V_4\rangle=6 g_1+10 g_3,\qquad g_{12}=0.$$
 Since $d_1=1$ and $a=1$,  we need
 $$
 (g_2+15 g_4)+(6 g_1+10 g_3)=\phi_1(t^2),\quad
(g_2+15 g_4)-(6 g_1+10 g_3)=t^2\,\phi_2(t^2)
 $$

For $\ell_2$ we have:
$$g_{11}=\langle\sqrt{10} V_2-\sqrt{6} V_7,\sqrt{10} V_2-\sqrt{6} V_7\rangle=10 g_2+6 g_4$$
$$g_{22}=\langle 4 V_5,4 V_5\rangle=16 g_3,\qquad g_{12}=0$$
Since $d_2=2$, smoothness requires that
$$(10g_2+6 g_4)+16 g_3=\phi_3(t^2),\quad
(10 g_2+6 g_4)-16 g_3=t^4\,\phi_4(t^2)$$

For $\ell_3$ we have:
$$g_{11}=\langle-\sqrt{15} V_3+ V_6,-\sqrt{15} V_3+ V_6\rangle=15 g_2+g_4$$
$$g_{22}=\langle\sqrt{10} V_1-\sqrt{6} V_4,\sqrt{10} V_1-\sqrt{6} V_4\rangle=10 g_1+6 g_3,\qquad g_{12}=0$$
Since $d_3=3$, we  need
$$(15g_2+  g_4)+(10 g_1+6 g_3)=\phi_5(t^2),\quad
(15 g_2+ g_4)-(10 g_1+6 g_3)=t^6\,\phi_6(t^2).$$

In particular, all functions $g_1,g_2,g_3,g_4$ are even, a fact that one could have already obtained from invariance of the metric under the Weyl group element.

For $\ell_0$, \lref{irredtrivial} says that $\langle\sqrt{6} V_2+\sqrt{10} V_7,\sqrt{6} V_2+\sqrt{10} V_7\rangle= 6 g_2+10g_4 $  is even, a condition that is already implied by the previous ones.

\bigskip

{ 2) \it Products between modules in $\mathfrak{m}$.}
Inner products between $\ell_0$ and $\ell_2$, and between $\ell_1$ and $\ell_3$ are not necessarily $0$. For the first one,  \lref{irredtrivial} implies that
 $$\langle\sqrt{6} V_2+\sqrt{10} V_7,\sqrt{10} V_2-\sqrt{6} V_7\rangle =\sqrt{60} (g_2- g_4)$$
 and hence $g_2- g_4=t^2\,\phi_7(t^2)$, a condition already implied by $K$ invariance at $t=0$.

 For the second one, \lref{irredproducts} and
 $$\langle V_3+\sqrt{15} V_6,-\sqrt{15} V_3+ V_6\rangle =\sqrt{15} (g_4-g_2),$$
$$\langle -\sqrt{6} V_1-\sqrt{10} V_4,\sqrt{10} V_1-\sqrt{6} V_4 \rangle =\sqrt{60} (g_3- g_1)$$
 as well as $$\langle V_3+\sqrt{15} V_6,\sqrt{10} V_1-\sqrt{6} V_4 \rangle= \langle -\sqrt{6} V_1-\sqrt{10} V_4,-\sqrt{15} V_3+ V_6 \rangle=0$$
 implies that
 $$
(g_4-g_2)-2 (g_3-g_1)=t^4\,\phi_7(t^2), \ (g_4-g_2)+ 2 (g_3-g_1)=t^2 \phi_8(t^2).
 $$

\bigskip

{ 3) \it Smoothness on the slice.}
Section \ref{slice} implies that $f=t^2+t^4\phi(t^2)$ since $a=1$.

\bigskip

{ 4) \it Products between $\mathfrak{m}$ and the slice $V$.}
All of the modules $\mathfrak{l}_i$ have  nontrivial inner products with the slice. For the 4
inner products between $\ell_{-1}'$, i.e. $K_2$, and $\ell_i$ we get from \lref{sliceX}:
$$
h_{11}=t^2\phi(t^2), \  h_{12}=t^3\phi(t^2), \ h_{11}=t^4\phi(t^2), \ h_{12}=t^5\phi(t^2).
$$

On the other hand, for the 4 inner products between $\ell_0'$, i.e. $K_3$, and $\ell_i$  we get from \lref{slicetrivial}:
$$
h_{12}=t^3\phi(t^2), \  h_{11}=t^2\phi(t^2), \ h_{12}=t^3\phi(t^2), \ h_{11}=t^4\phi(t^2).
$$
Thus we need:
$$
h_{11}=t^4\phi(t^2), \ h_{12}=t^5\phi(t^2).
$$

\bigskip

{ 5)\it Combining all conditions.}
Summarizing the conditions in 1) and 2), we have for the inner products in $\mathfrak{m}$:
$$\left\{
\begin{array}{l}
(g_2+15 g_4)+(6 g_1+10 g_3)=\phi_1(t^2)\\
(g_2+15 g_4)-(6 g_1+10 g_3)=t^2\,\phi_2(t^2)\\
(10g_2+6 g_4)+(16 g_3)=\phi_3(t^2)\\
(10 g_2+6 g_4)-(16 g_3)=t^4\,\phi_4(t^2)\\
(15g_2+  g_4)+(10 g_1+6 g_3)=\phi_5(t^2)\\
(15 g_2+ g_4)-(10 g_1+6 g_3)=t^6\,\phi_6(t^2)\\
(g_4-g_2)-2 (g_3-g_1)=t^4\,\phi_7(t^2)\\
g_2- g_4=t^2\,\phi_8(t^2), \ g_2- g_3=t^2\,\phi_9(t^2)\\
\end{array}\right.
$$
Notice that the last two conditions imply that $(g_3-g_1)=t^2\,\phi(t^2)$ and hence $K$ invariance at $t=0$ is  encoded in the above equations.

This is an over determined linear system of equations in the metric functions. Since we know there always exist solutions, we can row reduce in order to get the following relationships between the smooth functions:
$$\left\{\begin{array}{l}
\phi_1 = \phi_5-16 t^2\, \phi_7-2 t^4\,\phi_8\\
\phi_2 = t^4\,\phi_6-12\, \phi_7+2 t^2\,\phi_8\\
\phi_3= \phi_5-10 t^2\,\phi_7-5 t^4\,\phi_8\\
\phi_4=t^2\,\phi_6+5\,\phi_8
\end{array}\right.$$

Thus necessary and sufficient conditions for smoothness in $\mathfrak{m}$ are:
$$\left\{\begin{array}{l}
15 g_2+6 g_3+g_4+10g_1=\phi_5\\
15 g_2-6 g_3+g_4-10g_1=t^6\phi_6\\
-2 g_3+2g_1=t^2\widetilde{\phi}_7\\
-g_2-2 g_3+g_4+2g_1=t^4\phi_8\\
\end{array}\right.$$

which we can also solve for the metric and obtain (after renaming the even functions):

\begin{eqnarray*}
g_1 &=&\phi_1+6 t^2\,\phi_2+6 t^4 \,\phi_3 -t^6\,\phi_4\\
g_2 &=&\phi_1+2 t^2\,\phi_2+t^6\,\phi_4\\
g_3 &=&\phi_1-10 t^2\,\phi_2-10 t^4\,\phi_3- t^6\,\phi_4\\
g_4 &=&\phi_1-30 t^2\,\phi_2+ t^6\,\phi_4 
\end{eqnarray*}

for some smooth functions $\phi_1,\phi_2,\phi_3,\phi_4$ of $t^2$. Furthermore,
$$
f=t^2+t^4\phi_5(t^2), \ h_{11}=t^4\phi_6(t^2), \ h_{12}=t^5\phi_7(t^2)
$$
 \bigskip

 \begin{center} {\bf Example 4}\end{center}.

 This example shows how to predict the exponents $d_k$ in terms of representation theory. Let $\phi_n $ be the complex $n$-dimensional irreducible representation of $SU(2)$. Choose $K=SU(2)\subset G=SU(2n)$ given by the embedding $\phi_{2n}$, and $H=SO(2)=\diag(e^{i\theta},e^{-i\theta}) \subset SU(2)$. Thus $K/H=\Sph^2$ with slice representation $\phi_3$ and hence $a=2$. By Clebsch-Gordon, the isotropy representation of $G/K$ is $\phi_{4n-2}\oplus\phi_{4n-4}\oplus\cdots\oplus\phi_2$. Thus the isotropy representation $G/H$ is the sum of 2 dimensional representations $\mathfrak{n}_i$ with multiplicity $i$ and weight $4n-2i$   for $i=1,\cdots,2n-2$ and $\mathfrak{n}_{2n-1}$ and $\mathfrak{n}_{2n}$ with multiplicity $2n-2$ and weight $2$ resp. $0$, as well as $\mathfrak{n}_{2n+1}$ with weight $2$ coming from the isotropy representation of $K/H$.

 We only need to consider the one parameter group $L=\exp(tA)$ with $A=\left(
                                   \begin{array}{cc}
                                       0 & 1 \\
                                       -1 & 0 \\
                                     \end{array}
                                   \right)$. Since $A$ is conjugate to $\diag(i,-i)$, the decomposition under $L$ has the same weights and multiplicity. Thus in the description of the metric, we have exponents $t^k$ for $k=1,\cdots,4n-2$.

\bigskip

%%%%%%%%%%%%%%%%%%%%%%%%%%%%%%
%%%%%%%%%%%%%%%%%%%%%%%%%%%%%%%
\section{Proof of Theorem B}%%%%%%%%%%%
%%%%%%%%%%%%%%%%%%%%%%%%%%%%%%
%%%%%%%%%%%%%%%%%%%%%%%%%%%%%%%

 \bigskip

After we saw how the process works in concrete examples, we will now prove Theorem B.
One needs to first derive all smoothness conditions obtained from Section 3, possibly for several circles $L_i$. This gives rise to a highly over determined system of equations for the $r$ metric functions $g_{ij},\ i\le j$ of the form
$$
\sum_{i,j} a_{ij}^k\,g_{ij}(t)=t^{d_k}\phi_k(t^2),\quad k=1,\cdots, N.
$$
for some smooth functions $\phi_k$.

The coefficients $a_{ij}^k$ do not depend on the metric, but only on the Lie groups involved.
We first want to show that each metric function must be involved in at least one equation and hence  $N\ge r$. For this let $w\in K$ be a Weyl group element. Recall that $w$ is defined by $w(\dot c(0))=-\do \dot c(0)$ which defines it uniquely mod $H$. Furthermore, $w$ normalizes $H$ and $w^2\in H$. Let $\mathfrak{n}\subset \mathfrak{p}\oplus\mathfrak{m}$ be an irreducible module under the action of $H$. Then we have either $w(\mathfrak{n})=\mathfrak{n}$ or  $w(\mathfrak{n})=\mathfrak{n}'$ with $\mathfrak{n}'$ another irreducible module invariant under $H$ and equivalent to $\mathfrak{n}$.

 If  $w(\mathfrak{n})=\mathfrak{n}$ and $X,Y\in\mathfrak{n}$ then $Q( X,Y) =Q( wX,wY) $  and hence
$$
g( X^*,Y^*)_{c(t)}= g( (wX)^*, (wY)^*)_{c(t)}=g( X^*,Y^*)_{c(-t)}
$$
implies that $g( X^*,Y^*)$ is an even function.

If $w\mathfrak{n}=\mathfrak{n}'$ with $X\in\mathfrak{n},\ Y\in\mathfrak{n}'$, then we have
$$
g( X^*,Y^*)_{c(t)}- g( (wX)^*,(wY)^*)_{c(t)}= g( (wX)^*,(wY)^*)_{c(-t)}-g( X^*,Y^*)_{c(-t)}
$$
since $w^2\in H$. Thus $g( X^*,Y^*)- g( (wX)^*,(wY)^*)$ is an odd function, and similarly  $g( X^*,Y^*)+ g( (wX)^*,(wY)^*)$ is an even function.  Altogether, $N\ge r$.

We can now row reduce the systems, which we denote for short $A_kG=\Phi$. The last $N-r$ rows in $A_k$ will consist of zeroes which implies that there exists  linear homogeneous relationship between the even functions $\phi_k$. Solving for one of the variables, and substituting  into $\Phi$ we obtain a system of $r$ equations in $r$ unknowns. In the row-reduced system we cannot have a further row of zeroes in $A_k$ since otherwise we can express the metric in terms of $r-1$ even functions, contradicting that the metric on the regular part consists of $r$ arbitrary functions. Thus $A_k$ has maximal rank $r$ and we can solve for $g_{ij}$ in terms of the remaining even functions. This proves Theorem B.

\bigskip

 \section{Actions on Spheres}

 \bigskip

 In order to facilitate the applications of  determining the smoothness conditions in examples, we discuss here the a choice for the vectors  $X$, the decomposition of the action by $L=\exp(tX)$ on the slice, and the integers $a, d_i'$. Since $L\subset K_0$, we can assume that $K$ is connected. Although the action of $K$ on $V$ can be highly ineffective, there exists a normal subgroup containing  $L$ acting almost effectively and transitively on the sphere in $V$. In Table A we list the almost effective transitive actions by connected Lie groups on spheres. The effective actions and the decomposition of $\mathfrak{p}$ into irreducibles one can e.g. find in \cite{18}, and from this one easily derives the ineffective ones using representation theory.

 Recall that the inclusion $\mathfrak{p}\subset V$ is determined by the action fields of the action of $K$ on $V$. For each irreducible module we choose a vector $X\in\mathfrak{p}_i$  and  normalize $X$ such that $L=\{\exp(\gth v)\mid 0\le \gth \le 2\pi\}\subset K$ is a closed one parameter group. Furthermore, the integer $a=|L\cap H|$ is the ineffective kernel of the action of $L$ on $V$ and $V$
is the sum of two dimensional $L$ invariant  modules:
  $$V=\ell'_{-1}\oplus\ell_0'\oplus\ell_1'\oplus\cdots\oplus\ell_s' \text{ with } \ell'_{-1}=\spam\{\dot c(0),X\}$$    and
  $$L_{|\ell_{-1}'}=R(a\theta),\ L_{|\ell'_0}=\Ide \text{ and } L_{|\ell_i'}=R(d_i'\theta).$$
with $a,d_i'\in \mathbb{Z}$, which we can assume to be positive.

We choose a basis $e_1,e_2,\cdots  $ of $V$ and the geodesics $c(t)=te_1$.

\bigskip

We now discuss each transitive action, one at a time, using the numbering in Table A.
  \bigskip

  1) $\mathbf{K/H=SO(n+1)/SO(n)=\Sph^n}$

  \bigskip

 $\KGR$  acts by matrix multiplication $x\to Ax$ on $V=\mathbb{R}^{n+1}$ with orthonormal basis $e_1,e_2,\dots,e_{n+1}$. We choose the geodesic such that $c(t)=te_1$  and let $\HGR$ be the stabilizer group of $e_1$, i.e. $H=\{\diag(1,A)\mid A\in \SOG(n)\}$.

  As usual, we  use the notation $E_{ij}$ for the skew symmetric matrix with non-zero entries in the $(i,j)$ and $(j,i)$ spot and biinvariant inner product $Q(A,B)=-\frac 12 \trace(AB)$. Then $\mathfrak{p}=\spam\{E_{12},\dots E_{1(n+1)}\}$ and for  the action fields we get $E_{1i}^*=e_{i}$.

  We choose the closed one-parameter group $L=\{\exp(\theta E_{12})\mid 0\le\theta\le 2\pi\}$ which induces a rotation $R(\theta)$ in the $e_1,e_2$ plane. Thus
   $$L=\{\exp(\theta E_{12})\mid 0\le\theta\le 2\pi\}$$
  $$\ell_{-1}'=\{\dot c(0),E_{12}\} \text{ with } a=1,\  \text{ and } \ell_0'=\{E_{13},\dots E_{1(n+1)}\} $$

 \bigskip

  1') $\mathbf{K/H=Spin(n+1)/Spin(n)=\Sph^n}$

  \bigskip

$Spin(n+1)$ acts via the two fold cover $Spin(n+1)\to SO(n+1)$ ineffectively on $V$.
Since $L\subset SO(n+1)$ is a generator in $\pi_1(SO(n+1))\simeq\mathbb{Z}_2$, the lift of $L\subset SO(n+1)$ to $Spin(n+1)$ has twice its length. Thus, if $\bar E_{12}$ is the lift of $E_{12}$, the one parameter group $L=\{\exp(\theta \bar E_{12})\mid \theta\in\mathbb{R}\}$ induces a rotation $R(2\theta)$ in the $e_0,e_1$ plane. Hence $\ell_{-1}'=\{\dot c(0),\bar E_{12}\} \text{ with } a=2$ and $\ell_0'$ as before.

\bigskip

2) $\mathbf{K/H=U(n+1)/U(n)=\Sph^{2n+1}}$

\bigskip

$\KGR$  acts by matrix multiplication $x\to Ax$ on $V=\mathbb{C}^{n+1}$ with orthonormal basis $e_1,ie_1,\dots,e_{n+1},$ $ie_{n+1}$. $\HGR$ is the stabilizer of $e_1$, i.e. $\HGR=\ULG(n)=\{\diag(1,A)\mid A\in \ULG(n)\}$.  Besides $E_{ij}$, we have the skew hermitian matrix $iE_{ij}$ (by abuse of notation). We use the inner product $Q(A,B)=-\frac 12 Re ( \trace(AB))$, and hence $\mathfrak{p}=\mathfrak{p}_0\oplus\mathfrak{p}_1$ with $\mathfrak{p}_0=\mathbb{R}\cdot F$ with $F= \diag(i,0,\cdots,0)$, $\mathfrak{p}_1=\spam\{E_{12},iE_{12},\dots E_{1(n+1)},iE_{1(n+1)}\}$. For the action fields we have  $F^*=ie_1$ and  $E_{1i}^*=e_{i},\ iE_{1i}^*=ie_{i},\ i=2,\cdots,n+1$.

We need to choose two closed one parameter subgroups,   $L_1=\{\exp(\theta E_{12})$ and $L_2=\exp(\theta F)$ with $0\le\theta\le2\pi$.

\smallskip

 $L_1$  induces a rotation $R(\theta)$ in the $e_1,e_2$ plane, and in the $ie_1, ie_2$ plane as well. Thus

$$
L_1=\{\exp(\theta E_{12})\mid 0\le\theta\le 2\pi\}
$$
we have $\ell_{-1}'=\{\dot c(0),E_{12}\}$, with $a=1$, $\ell_1'=\{ F,iE_{12}\}$, with  $d_1'=1$, and $\ell_0'=\{E_{1r},iE_{1r},r\ge 3\}$.

\bigskip

Next, $L_2=\{\exp(\theta F)\mid 0\le\theta\le 2\pi\}$  induces a rotation $R(\theta)$ in the $e_1,ie_1$ plane, and as $\Ide$ on the rest. Thus

$$L_2=\{\exp(\theta F)\mid 0\le\theta\le 2\pi\}$$
$$\ell_{-1}'=\{\dot c(0), F\}, \text{ with } a=1,\   \text{ and } \ell_0'=\{E_{1r},iE_{1r},\ r\ge 2\}$$

\bigskip

2') $\mathbf{K/H=U(n+1)/U(n)_k=\Sph^{2n+1}}$

\bigskip

In this case $\ULG(n+1)$ acts as $v\to (\det A)^k Av$ for some integer $k\ge 1$, and hence the stabilizer group of $e_1$ is $\HGR=\SUG(n)\cdot\S^1_k$ with $\S^1_k=\diag( z^{nk},\bar z^{k+1}, \cdots, \bar z^{k+1} )$. Thus we have $\mathfrak{p}=\mathfrak{p}_0\oplus\mathfrak{p}_1$ as in case 2),  but now $\mathfrak{p}_0=\mathbb{R}\cdot F$ with $F= \diag( (k+1)i, ki, \cdots,  ki )$ and hence $F^*=(k+1)ie_1$.

\smallskip

The case of  $L_1=\exp(\theta E_{12})$ is as in the previous  case, except that
$\ell_1'=\{\frac1{k+1} F,iE_{12}\}$.

\smallskip

But now  $L_2=\{\exp(\theta F)\mid 0\le\theta\le 2\pi\}$ acts as $R((k+1)\theta)$ in the $e_1,ie_1$ plane, and $R(k\theta)$ in the $e_r, ie_r$ plane, $r\ge 2$. Hence

$$L_2=\{\exp(\theta F)\mid 0\le\theta\le 2\pi\}$$
$$\ell_{-1}'=\{\dot c(0),\frac1{k+1}  F\}, \text{ with } a=k+1, \ \text{ and }\ell_r'=\{E_{1r},iE_{1r}\},r\ge 2,  \text{ with } d'_r= k. $$

\bigskip

 2') $\mathbf{K/H=U(1)/Z_k=\Sph^{2n+1}}$

\bigskip

We list here separately the common case of $K=U(1)$ acting on $\mathbb{C}$ as $w\to z^kw$ with stabilizer group $Z_k$ the k-th roots of unity. Here $\mathfrak{p}=\mathfrak{p}_0$ spanned by $F=i$ with $F^*=kie_1$. Thus
$\ell_{-1}'=\{\dot c(0),\frac1{k}  F\}$ with $a=k$.

\bigskip

    3) $\mathbf{K/H=SU(n+1)/SU(n)=\Sph^{2n+1}}$

    \bigskip

    Same action and basis as in case 2, with $\HGR=\SUG(n)=\{\diag(1,A)\mid A\in \SUG(n)\}$.   But now $F= \diag(ni,-i,\cdots,-i)$ and hence $F^*=nie_1$.

    \smallskip
    Thus the result for  $L_1=\{\exp(\theta E_{12})\mid 0\le\theta\le 2\pi\}$ is as before, except that $ \ell_1'=\{ F,iE_{12}\}$.

    \smallskip

  Now  $L_2=\{\exp(\theta F)\mid 0\le\theta\le 2\pi\}$  induces a rotation $R(n\theta)$ in the $e_1,ie_1$ plane, and $R(-\theta)$ in the $e_k, ie_k$ plane, $k\ge 2$. Thus

   $$L_2=\{\exp(\theta F)\mid 0\le\theta\le 2\pi\}$$
    $$\ell_{-1}'=\{\dot c(0),\frac1n F\}, \text{ with } a=n,\   \text{ and } \ell_r'=\{iE_{1r},E_{1r}\},r\ge 2,  \text{ with } d'_r=1$$

 \bigskip

       4) $\mathbf{K/H=Sp(n+1)/Sp(n)=\Sph^{4n+3}}$

  \bigskip

$\KGR$  acts by matrix multiplication $x\to Ax$ on $V=\mathbb{H}^{n+1}$, with orthonormal basis $e_0,ie_0,je_0,ke_0,\cdots $ and $\HGR$ is the stabilizer of $e_0$ i.e., $\HGR=\{\diag(1,A)\mid A\in \SpG(n)\}$, acting on $\mathfrak{p}=\Im\mathbb{H}\oplus \mathbb{H}^n$ as $(s,x) \to (s,Ax)$.
We  have the basis  of $\mathfrak{k}$ given by $E_{ij}, iE_{ij},jE_{ij},kE_{ij}$, where, by abuse of notation, the last three are skew hermitian, and $F_1=\diag(i,0,\cdots ,0),\ F_2=\diag(j,0,\cdots ,0),\ F_3=\diag(k,0,\cdots ,0)$. As before, $Q(A,B)=-\frac 12 Re(\trace(AB))$, and $\mathfrak{p}=\mathfrak{p}_0\oplus\mathfrak{p}_1$ with $\mathfrak{p}_0=\spam(F_1,F_2,F_3)$  and  $\mathfrak{p}_1=\spam\{E_{1r}, iE_{1r},jE_{1r},kE_{1r}, r=2,\cdots n+1\}$. For the action fields we have  $F_1^*=ie_1, F_2^*=je_1, F_3^*=ke_1$ and  $E_{1s}^*=e_{s},\ iE_{1s}^*=ie_{s},\ jE_{1s}^*(c(1)=je_{s},\ kE_{1s}^*=ke_{s},\ s=2,\cdots n+1, $.

 We need to consider four 1-parameter groups $L_1=\{\exp(\theta E_{12})\mid 0\le\theta\le 2\pi\}$,  $L_2=\exp(\theta F_1),\ L_3=\exp(\theta F_2)$ and $ L_4=\exp(\theta F_3)$ with $0\le\theta\le2\pi$.
\smallskip

 For $L_1$, acting on $V$, we get:

  $$L_1=\{\exp(\theta E_{12})\mid 0\le\theta\le 2\pi\}$$
we have $\ell_{-1}'=\{\dot c(0),E_{12}\}$, with $a=1$, $\ell_1'=\{  F_1,iE_{12}\}$,  $\ell_2'=\{  F_2,jE_{12}\}$,  $\ell_3'=\{  F_3,kE_{12}\}$ with $d'_r=1$,
and $\ell_0'=\{E_{1r},iE_{1r},jE_{1r},kE_{1r},\ r\ge 3\}$.

 \smallskip

The one parameter group  $L_2=\exp(\theta F_1)$ rotates the planes $e_1,ie_1$ and $je_1,ke_1$ by $R(\theta)$ and fixes all remaining vectors. Thus

$$L_2=\{\exp(\theta F_1)\mid 0\le\theta\le 2\pi\}$$
$$\ell_{-1}'=\{\dot c(0),F_1\}, \text{ with } a=1,\ \ell_1'=\{  F_2,F_3\}  \text{ with } d'_1=1 $$
 $$\text{ and }\ell_0'=\{E_{1r},iE_{1r},jE_{1r},kE_{1r},\ r\ge 2\}, $$
and similarly for $L_3, L_4$.

        \bigskip

        5) $\mathbf{K/H=Sp(n+1)\cdot Sp(1)/Sp(n)\cdot\Delta Sp(1)=\Sph^{4n+3}}$
 \bigskip

The slice is $V=\mathbb{H}^{n+1}$ with basis $e_1,ie_1,je_1ke_1,\cdots $ and $(A,q)\in\KGR$ acting as $v\to Avq^{-1}$. Here we are considering the effective action and thus $K=Sp(n+1)\times Sp(1)/\mathbb{Z}_2$ with $\mathbb{Z}_2=(-\Ide,-1)$. The stabilizer group of $e_1$ is $\HGR=\SpG(n)\Delta Sp(1)=\{(\diag(q,A),q)\mid A\in \SpG(n), q\in Sp(1)\}\simeq Sp(n)\times Sp(1)/\mathbb{Z}_2$ acting on $\mathfrak{p}=\Im\mathbb{H}\oplus \mathbb{H}^n$ as $(s,x) \to (qsq^{-1},Axq^{-1})$.
Again, $\mathfrak{p}=\mathfrak{p}_0\oplus\mathfrak{p}_1$ with $\mathfrak{p}_0=\spam(F_1,F_2,F_3)$  and  $\mathfrak{p}_1=\spam\{E_{1r}, iE_{1r},jE_{1r},kE_{1r}, r=2,\cdots n+1\}$, but now  $F_1=(\diag(i,0,\cdots ,0),-i)$, $F_2=(\diag(j,0,\cdots ,0),-j)$, $F_3=(\diag(k,0,\cdots ,0),-k)$ with $F_1^*=2ie_1, F_2^*=2je_1, F_3^*=2ke_1$.

    We  need to consider only two 1-parameter groups $L_1=  \{ (\exp(\theta E_{12}),1)\mid 0\le\theta\le 2\pi\}$ and  $L_2=\{\exp(\theta F_1)\mid 0\le\theta\le 2\pi\}$.
\smallskip

For $L_1=\exp(\theta E_{12})$ we get:

 $$L_1=  \{ (\exp(\theta E_{12}),1)\mid 0\le\theta\le 2\pi\}$$
 $\ell_{-1}'=\{\dot c(0),E_{12}\}$, with $a=1$, $\ell_1'=\{ \tfrac12 F_1,iE_{12}\}$, $\ell_2'=\{ \tfrac12 F_2,jE_{12}\}$, $\ell_3'=\{ \tfrac12 F_3,kE_{12}\}$ with $d'_r=1$
and $\ell_0'=\{E_{1r},iE_{1r},jE_{1r},kE_{1r},\ r\ge 3\}$.

 \smallskip

The one parameter group  $L_2$ rotates the planes $e_1,ie_1$ by $R(2\theta)$  and fixes all remaining vectors, including $F_2,F_3$. Thus

$$L_2=\{\exp(\theta F_1)\mid 0\le\theta\le 2\pi\}$$
$\ell_{-1}'=\{\dot c(0),\tfrac12 F_1\}$, with $a=2$, and $\ell_0'=\{\tfrac12 F_2,\tfrac12 F_3, E_{1r},iE_{1r},jE_{1r},kE_{1r},\ r\ge 2\}$.

\

        5') $\mathbf{K/H=Sp(n+1)\times Sp(1)/Sp(n)\times\Delta Sp(1)=\Sph^{4n+3}}$
  \smallskip

  The action is as in the previous case, but now with an ineffective kernel $\mathbb{Z}_2=(-\Ide,-1)\in Sp(n+1)\times Sp(1)$. The decompositions and the integers though are the same.

        \

         6) $\mathbf{K/H=Sp(n+1)U(1)/Sp(n)\Delta U(1)_k=\Sph^{4n+3}}$

           \smallskip

This case is similar to the previous one, but here
          $\KGR$ acts as $v\to Av\bar z^k$ on $V=\mathbb{H}^{n+1}$  with ineffective kernel $\{(\Ide,z)\mid z^k=1\}$. Furthermore, $\HGR=\SpG(n)\Delta U(1)_k=\{(\diag(z^k,A),z)\mid A\in \SpG(n), z\in U(1)\}$.
          If $F_1=(\diag(i,0,\cdots ,0),-ki),\ F_2=(\diag(j,0,\cdots ,0),0), \ F_3=(\diag(k,0,\cdots ,0),0)$, then
         $\mathfrak{p}=\mathfrak{p}_0\oplus\mathfrak{p}_1\oplus\mathfrak{p}_2$ with $\mathfrak{p}_0=\spam(F_1)$, $\mathfrak{p}_1=\spam(F_2,F_3)$,   and  $\mathfrak{p}_2=\spam\{E_{1r}, iE_{1r},jE_{1r},kE_{1r}, r=2,\cdots n+1\}$. Furthermore, $(A,z)\in H$ acts on $\mathfrak{p}$ as $(s,x)\to (z^ksz^{-k},Axz^{-1})$, where $s\in\Im\mathbb{H}=\mathfrak{p}_0\oplus\mathfrak{p}_1$. Notice that $H$ acts trivially on $\mathfrak{p}_0$  and that $F_1^*=(k+1)ie_1, F_2^*=je_1, F_3^*=ke_1$.
\smallskip

         We  need to consider the 1-parameter groups $L_1=(\exp(\theta E_{12}),1)$,  $L_2=\exp(\theta F_1)$ and $L_3=\exp(\theta F_2)$. For $L_1=\exp(\theta E_{12})$, similarly to case 6), we get:

   $$L_1=\{(\exp(\theta E_{12}),1)\mid 0\le\theta\le 2\pi\}$$
we have $\ell_{-1}'=\{\dot c(0),E_{12}\}$, with $a=1$, $\ell_1'=\{ \tfrac{1}{k+1} F_1,iE_{12}\},\ell_2'=\{  F_2,jE_{12}\},\ell_3'=\{  F_3,kE_{12}\}$ with $d'_r=1$
and $\ell_0'=\{E_{1r},iE_{1r},jE_{1r},kE_{1r},\ r\ge 3\}$. 

\smallskip

 For $L_2$ on the other hand, we have

 $$L_2=\{\exp(\theta F_1)\mid 0\le\theta\le 2\pi\}$$
$$\ell_{-1}'=\{\dot c(0),\tfrac12 F_1\}, \text{ with } a=k+1, \ell_1'=\{  F_3,F_2\}  \text{ with } d'_1=k-1$$
$$
\text{ and }\ell_0'=\{ E_{1r},iE_{1r},jE_{1r},kE_{1r},\ r\ge 2\}. $$
\smallskip
For $L_3$ we have:
$$L_3=\{\exp(\theta F_2)\mid 0\le\theta\le 2\pi\}$$
$$\ell_{-1}'=\{\dot c(0), F_2\}, \text{ with } a=1, \ell_1'=\{  F_3,F_3\}  \text{ with } d'_1=1$$
$$
\text{ and }\ell_0'=\{ E_{1r},iE_{1r},jE_{1r},kE_{1r},\ r\ge 2\}. $$

 \bigskip

         7) $\mathbf{K/H=G_2/SU(3)=\Sph^6}$

   \bigskip

  We regard $\GGR_2$ as the automorphism group of the Cayley numbers with basis
  $1,i,j,k,\ell, i\ell, j\ell, k\ell$. This embeds $\GGR_2$ naturally into $\SOG(7)$ and its action is transitive on $\Sph^6$. On the Lie algebra level, a skew symmetric matrix $(a_{ij})\in\mathfrak{so}(7)$ belongs to $\mathfrak{g}_2$ iff
  $$
  a_{23} + a_{45} + a_{76} = 0,\ a_{12} + a_{47} + a_{65} = 0,\ a_{13} + a_{64} + a_{75} = 0$$
$$a_{14} + a_{72} + a_{36} = 0,\
a_{15} + a_{26} + a_{37} = 0,\
a_{16} + a_{52} + a_{43} = 0,\
a_{17} + a_{24} + a_{53} = 0.
  $$

  Thus a basis for the Lie algebra $\mathfrak{g}_2\subset\mathfrak{so}(7)$ is given by
 $$ \left(
    \begin{array}{ccccccc}
      0 & x_1+x_2 & y_1+y_2& x_3+x_4 & y_3+y_4 & x_5+x_6 & y_5+y_6 \\
      -(x_1+x_2) & 0 & \alpha_1 & -y_5 & x_5 & -y_3 & x_3 \\
      -(y_1+y_2) & -\alpha_1 & 0 & x_6 & y_6 & -x_4 & -y_4 \\
      -(x_3+x_4) & y_5 & -x_6 & 0 & \alpha_2 & y_1 & -x_1 \\
      -(y_3+y_4) & -x_5 & -x_6 & -\alpha_2 & 0 & x_2 & y_2 \\
      -(x_5+x_6) & y_3 & x_4 & -y_1 & -x_2 & 0 & \alpha_1+\alpha_2 \\
      -(y_5+y_6) & -x_3 & y_4 & x_1 & -y_2 & -(\alpha_1+\alpha_2) & 0 \\
    \end{array}
  \right)
$$
The stabilizer group at $i$ is given by the complex linear automorphisms, which is equal to $\SUG(3)$. Thus its Lie algebra $\mathfrak{h}$ is given by the constraints $x_i+x_{i+1}=y_i+y_{i+1}=0$ for $i=1,3,5$, and
 the complement $\mathfrak{p}$ by
$$ \left(
    \begin{array}{ccccccc}
      0 & 2x_1 & 2y_1& 2x_3 & 2y_3 & 2x_5 & 2y_5 \\
      -2x_1 & 0 & 0 & -y_5 & x_5 & -y_3 & x_3 \\
      -2y_1 & 0 & 0 & x_5 & y_5 & -x_3 & -y_3 \\
      -2x_3 & y_5 & -x_5 & 0 & 0 & y_1 & -x_1 \\
      -2y_3 & -x_5 & -x_5 & 0 & 0 & x_1 & y_1 \\
      -2x_5 & y_3 & x_3 & -y_1 & -x_1 & 0 & 0 \\
      -2y_5 & -x_3 & y_3 & x_1 & -y_1 & 0 & 0 \\
    \end{array}
  \right)
$$
  Since the action of $\Adj_H$ on $\mathfrak{p}$ is irreducible,
  it is sufficient to consider only one one-parameter group, and we choose $F=2E_{12}-E_{47}+E_{56}\in \mathfrak{p}$ with $L=\exp(\theta F)$. It acts as a rotation in the $e_4,e_7$ plane and $e_5,e_6$ plane at speed $1$, and in the $e_1,e_2$ plane at speed 2, and as $\Ide$ on $e_3$.

Thus

$$L=\{\exp(\theta F)\mid 0\le \theta\le 2\pi\}$$
$$\ell_{-1}'=\{\dot c(0),F\} \text{ with } a=2,\   \ell_1'=\{  2E_{14}+E_{27}-E_{36},\ 2E_{17}+E_{35}-E_{24}\} \text{ with } d_1'=1
$$
$$\ell_2'=\{ 2E_{16}+E_{25}+E_{34},\ 2E_{15}-E_{26}-E_{37}\}
\text{ with } d_2'=1, \text{ and } \ell_0'=\{2E_{13}+E_{57}+E_{46}\}
$$

    \bigskip

         8) $\mathbf{K/H=Spin(7)/G_2=\Sph^7}$

  \bigskip

  The embedding $\Spin(7)\subset\SOG(8)$, and hence the action of $\KGR$ on the slice, is given by the spin representation. On the Lie algebra level we can describe this as follows. A basis of $\mathfrak{g}_2\subset\mathfrak{so}(8)$ is given by the span of
 \begin{align*}
& E_{24}+E_{68},\ E_{28}+E_{46},\ E_{26}-E_{48}
 E_{23}+E_{67},\ E_{27}+E_{36},\ E_{34}+E_{78}\\
&E_{38}+E_{47},\ E_{37}-E_{48},\ E_{27}-E_{45},\
E_{23}+E_{58},\ E_{24}-E_{57},\ E_{28}+E_{35}\\
& E_{56}-E_{78},\ 2\,\ E_{25}-E_{38}+E_{47}
\end{align*}
and the complement $\mathfrak{p}$ by the span of

$$ E_{12}+E_{56},\ E_{13}+E_{57},\ E_{14}+E_{58},\ E_{15}-E_{48} ,\ E_{16}+E_{25},\
E_{17}+E_{35},\ E_{18}+E_{45}.
$$
Since the action of $\Adj_H$ on $\mathfrak{p}$ is irreducible, we need to consider only one one-parameter group and we choose $L=\{\exp(\theta F)$ with $F=E_{12}+E_{56}$. It acts as a rotation in the $e_1,e_2$ plane and $e_5,e_6$ plane at speed $1$, and  as $\Ide$ on $e_3,e_4,e_7, e_8$.

 Thus
$$L=\{\exp(\theta F)\mid 0\le \theta\le 2\pi\}$$
$$\ell_{-1}'=\{\dot c(0),F\} \text{ with } a=1,\   \ell_1'=\{  E_{15}-E_{48},\ E_{16}+E_{25}\} \text{ with } d_1'=1
$$
$$
 \text{ and } \ell_0'=\{ E_{13}+E_{57},\ E_{14}+E_{58},\
E_{17}+E_{35},\ E_{18}+E_{45}  \}.
$$

    \bigskip

         9) $\mathbf{K/H=Spin(9)/Spin(7)=\Sph^{15}}$

  \bigskip

  The embedding of $\HGR$ in $\KGR$ is given by the spin representation of $\Spin(7)$ in $\Spin(8)$ followed by the (lift of) the standard block embedding of $\Spin(8)$ in $\Spin(9)$.
  Let $S_{ij}$ be the  standard basis of $\mathfrak{spin}(9)$ under the isomorphism $\mathfrak{so}(9)\simeq\mathfrak{spin}(9)$  and denote by $E_{i,j}$ the standard basis of $\mathfrak{so}(16)$.  Furthermore, $\Spin(9)$ acts on the slice $V\simeq\mathbb{R}^{16}$ via the spin representation and one easily computes the image of $S_{ij}$ in $\mathfrak{so}(16)$. We only need the basis of $\mathfrak{p}=\mathfrak{p}_1\oplus\mathfrak{p}_2$.

The irreducible $7$-dimensional module $\mathfrak{p}_1$ is spanned by
\begin{align*}
Z_2:&=&-S_{78}+S_{12}+S_{34}+S_{56}&=&2 E_{1, 2}+E_{9, 10}+E_{11, 12}+E_{13, 14}-E_{15, 16}\\
Z_3:&=&S_{68}+S_{13}-S_{24}+S_{57}&=&2 E_{1, 3}+E_{9, 11}-E_{10, 12}+E_{13, 15}+E_{14, 16}\\
Z_4:&=&S_{58}+S_{14}+S_{23}-S_{67}&=&2 E_{1, 4}+E_{9, 12}+E_{10, 11}+E_{13, 16}-E_{14, 15} \\
Z_5:&=&-S_{48}+S_{15}-S_{26}-S_{37}&=&2 E_{1, 5}+E_{9, 13}-E_{10, 14}-E_{11, 15}-E_{12, 16}\\
Z_6:&=&-S_{38}+S_{16}+S_{25}+S_{47}&=&2 E_{1, 6}+E_{9, 14}+E_{10, 13}-E_{11, 16}+E_{12, 15} \\
Z_7:&=& S_{28}+S_{17}+S_{35}-S_{46}&=&2 E_{1, 7}+E_{9, 15}+E_{10, 16}+E_{11, 13}-E_{12, 14}\\
Z_8:&=&S_{18}-S_{27}+S_{36}+S_{45}&=&2 E_{1, 8}+E_{9, 16}-E_{10, 15}+E_{11, 14}+E_{12, 13}
\end{align*}

and the irreducible $8$-dimensional module $\mathfrak{p}_2$ is spanned by
$S_{i,9}$

$$\begin{array}{l}
S_{19}=\frac 12(E_{1, 9}+E_{2, 10}+E_{3, 11}+E_{4, 12}+E_{5, 13}+E_{6, 14}+E_{7, 15}+E_{8, 16})\\
S_{29}=\frac 12(E_{1, 10}-E_{2, 9}-E_{3, 12}+E_{4, 11}-E_{5, 14}+E_{6, 13}+E_{7, 16}-E_{8, 15})\\
S_{39}=\frac 12(E_{1, 11}+E_{2, 12}-E_{3, 9}-E_{4, 10}-E_{5, 15}-E_{6, 16}+E_{7, 13}+E_{8, 14})\\
S_{49}=\frac 12(E_{1, 12}-E_{2, 11}+E_{3, 10}-E_{4, 9}-E_{5, 16}+E_{6, 15}-E_{7, 14}+E_{8, 13})\\
S_{59}=\frac 12(E_{1, 13}+E_{2, 14}+E_{3, 15}+E_{4, 16}-E_{5, 9}-E_{6, 10}-E_{7, 11}-E_{8, 12})\\
S_{69}=\frac 12(E_{1, 14}-E_{2, 13}+E_{3, 16}-E_{4, 15}+E_{5, 10}-E_{6, 9}+E_{7, 12}-E_{8, 11})\\
S_{79}=\frac 12(E_{1, 15}-E_{2, 16}-E_{3, 13}+E_{4, 14}+E_{5, 11}-E_{6, 12}-E_{7, 9}+E_{8, 10})\\
S_{89}=\frac 12(E_{1, 16}+E_{2, 15}-E_{3, 14}-E_{4, 13}+E_{5, 12}+E_{6, 11}-E_{7, 10}-E_{8, 9})
\end{array}$$

If $e_1,\cdots,e_{16}$ is a basis of the slice, then $Z_i^*=e_i,\ i=2,\cdots,8$ and $S_{i9}^*=e_{i+8},\ i=1,\cdots,8$.

For the smoothness conditions we need to choose two one parameter groups. For $L_1=\exp(\theta Z_2)$  we obtain
$$L_1=\{\exp(\theta Z_2)\mid 0\le \theta\le 2\pi\}$$
$\ell_{-1}'=\{\dot c(0),Z_2\}$ with $a=2$, $\ell_i'=\{  S_{i,9},\ S_{i+1,9} \}$, $i=1,3,5,7$  with $d_i'=1$ for $i=1,3,5$,  $d_7'=-1$ and $\ell_0'=\{ Z_3,\cdots, Z_8  \}$.

In $\ell_7'$ we should reverse the order of the basis so that $d_7'=1$.
\smallskip

For  $L_2=\exp(\theta S_{19})$ we have
$$L_2=\{\exp(\theta S_{19})\mid 0\le \theta\le 2\pi\}$$
$$\ell_{-1}'=\{\dot c(0),S_{19}\} \text{ with } a=1,\   \ell_i'=\{ Z_i, S_{i,9} \},\ i=2,\cdots 8\text{ with } d_i'=1
$$

 \providecommand{\bysame}{\leavevmode\hbox
to3em{\hrulefill}\thinspace}

\renewcommand{\thetable}{\Alph{table}}
\renewcommand{\arraystretch}{1.4}
\stepcounter{equation}
\begin{table}[!ht]
     \begin{center}
         \begin{tabular}{|c|c|c|c|c|c|c|c|c|}
\hline
&K &H & $\mathfrak{p}_i$ &$\dim \mathfrak{p}_i$\\ \hline \hline
1&$\SOG(n+1$) &$SO(n)$ &$\mathfrak{p}_1$ &n\\ \hline
1'&$\Spin(n+1$) &$Spin(n)$ &$\mathfrak{p}_1$ &n\\ \hline
2&$\ULG(n+1)$ &$\ULG(n)$ &$\mathfrak{p}_0+\mathfrak{p}_1$ &1, 2n\\ \hline
2'&$\ULG(n+1)$ &$\ULG(n)_k$ &$\mathfrak{p}_0+\mathfrak{p}_1$ &1, 2n\\ \hline
3&$\SUG(n+1)$ &$\SUG(n)$ &$\mathfrak{p}_0+\mathfrak{p}_1$ &1, 2n\\ \hline
4&$\SpG(n+1)$ &$\SpG(n)$ &$\mathfrak{p}_0+\mathfrak{p}_1$ &3, 4n\\ \hline
5&$\SpG(n+1)\cdot \SpG(1)$ &$\SpG(n)\Delta \SpG(1)$ &$\mathfrak{p}_1+\mathfrak{p}_2$ &3, 4n\\ \hline
5'&$\SpG(n+1)\times \SpG(1)$ &$\SpG(n)\times \Delta \SpG(1)$ &$\mathfrak{p}_1+\mathfrak{p}_2$ &3, 4n\\ \hline
6&$\SpG(n+1)\cdot \ULG(1)$ &$\SpG(n)\Delta\ULG(1)$ &$\mathfrak{p}_0+\mathfrak{p}_1+\mathfrak{p}_2$ &1, 2, 4n\\ \hline
6'&$\SpG(n+1)\times \ULG(1)$ &$\SpG(n)\Delta \ULG(1)_k$ &$\mathfrak{p}_0+\mathfrak{p}_1+\mathfrak{p}_2$ &1, 2, 4n\\ \hline
7&$\GGR_2$ &$\SUG(3)$ &$\mathfrak{p}_1$\ &6\\ \hline
8&$\Spin(7)$ &$\GGR_2$ &$\mathfrak{p}_1$&7\\ \hline
9&$\Spin(9)$ &$\Spin(7)$ &$\mathfrak{p}_1+\mathfrak{p}_2$ &8, 7\\
\hline
         \end{tabular}
     \end{center}
     \vspace{0.1cm}
     \caption{Almost effective transitive actions on spheres}\label{transitive}
\end{table}

\renewcommand{\arraystretch}{1.8}
\begin{table}[!ht]
     \begin{center}
         \begin{tabular}{|c|c|c|c|c|c|c|c|c|}
\hline
$\langle \mathfrak{m},\mathfrak{m}\rangle$ & $\ell_0$ & $\ell_i$ & $\ell_j$\\ \hline \hline
$\ell_0$&$\phi(t^2)$ &$ t^{\frac{d_i}{a}}\phi(t^2)$ & $ t^{\frac{d_j}{a}}\phi(t^2)$ \\ \hline
$\ell_i$&$ t^{\frac{d_i}{a}}\phi(t^2)$ &
\begin{minipage}{99pt}\vspace{5pt}
$g_{11}+g_{22}=\phi_1(t^2)$\\
$g_{11}-g_{22}=t^{\frac{2d_i}{a}}\phi_2(t^2)$\\
$\hspace{10pt} \ g_{12}=t^{\frac{2d_i}{a}}\phi_3(t^2)$ \vspace{1pt}
\end{minipage}
  & \begin{minipage}{130pt}
$h_{11}+h_{22}=t^{\frac{|d_i-d_j|}{a} }\phi_1(t^2),\\  h_{11}-h_{22}=t^{\frac{|d_i+d_j|}{a} }\phi_1(t^2)$\\
$h_{12}-h_{21}=t^{\frac{|d_i-d_j|}{a} }\phi_1(t^2),\\  h_{12}+h_{22}=t^{\frac{|d_i+d_j|}{a} }\phi_1(t^2)$
\end{minipage}\\
\hline
         \end{tabular}
     \end{center}
     \vspace{0.1cm}
     \caption{Smoothness Conditions I for $G$ invariant metrics or symmetric $2\times 2$ tensors}\label{smooth3}
\end{table}

\renewcommand{\arraystretch}{1.8}
\begin{table}[!ht]
     \begin{center}
         \begin{tabular}{|c|c|c|c|}
\hline
$\langle \mathfrak{p},\mathfrak{m}\rangle$ & $\ell_0$ & $\ell_j$ \\ \hline \hline
$\ell_{-1}'$ &  $t^2\phi(t^2)$ & $t^2 t^{\frac{d_j}{a}}\phi(t^2)$  \\ \hline
$\ell_0'$&$ t^3\phi(t^2)$ &
$t\, t^{\frac{d_j}{a}}\phi(t^2)$ \\ \hline
$\ell_i'$ & $t\, t^{\frac{d_i'}{a}}\phi(t^2)$ &
\begin{minipage}{230pt}\vspace{5pt}
$h_{11}+h_{22}=t^bt^{\frac{|d_i'-d_j|}{a} }\phi_1(t^2),\\ h_{11}-h_{22}=t\,t^{\frac{|d_i'+d_j|}{a} }\phi_2(t^2)$\\
$h_{12}-h_{21}=t^bt^{\frac{|d_i'-d_j|}{a} }\phi_3(t^2),\\ h_{12}+h_{21}=t\,t^{\frac{|d_i'+d_j|}{a} }\phi_4(t^2)$\vspace{5pt}
\\
 $b=3$ if $d_i'=d_j$, and $b=1$ if $d_i'\ne d_j$
\\
\end{minipage} \\
\hline
         \end{tabular}
     \end{center}
     \vspace{0.1cm}
     \caption{Smoothness Conditions II for $G$ invariant metrics}\label{smooth1}
\end{table}

\

\renewcommand{\arraystretch}{1.8}
\begin{table}[!ht]
	\begin{center}
		\begin{tabular}{|c|c|c|c|}
			\hline
			$\langle \mathfrak{p},\mathfrak{m}\rangle$ &  $\ell_0$ & $\ell_j$ \\ \hline \hline
			$\ell_{-1}'$  &
			$t^2\phi(t^2)$ &
		    \begin{minipage}{230pt}\vspace{5pt}
			$ T(\dot c, Y_1^*)+T(X^*,Y_2^*)=t^{\frac{d_j}{a}}\phi_1(t^2),\\ T(\dot c, Y_1^*)-T(X^*,Y_2^*)=t^2t^{\frac{d_j}{a}} \phi_2(t^2)$
			\\
			$ T(\dot c, Y_2^*)-T(X^*,Y_1^*)=t^{\frac{d_j}{a}}\phi_1(t^2),\\ T(\dot c, Y_2^*)+T(X^*,Y_1^*)=t^2t^{\frac{d_j}{a}} \phi_2(t^2)$
			\\
			\end{minipage}
			\\
			\hline
			$\ell_0'$ &$ t\,\phi(t^2)$ &
			$t\, t^{\frac{d_j}{a}}\phi(t^2)$ \\ \hline
			$\ell_i'$ & $t\, t^{\frac{d_i'}{a}}\phi(t^2)$ &
			\begin{minipage}{180pt}\vspace{5pt}
				$T_{11}+T_{22}=t\, t^{\frac{|d_i'-d_j|}{a} }\phi_1(t^2),\\ T_{11}-T_{22}=t\,t^{\frac{|d_i'+d_j|}{a} }\phi_2(t^2)$\\
				$T_{12}-T_{21}=t\,t^{\frac{|d_i'-d_j|}{a} }\phi_3(t^2),\\ T_{12}+T_{21}=t\,t^{\frac{|d_i'+d_j|}{a} }\phi_4(t^2)$
				\\
			\end{minipage}
			\\
			\hline
			
		\end{tabular}
	\end{center}
	\vspace{0.1cm}
	\caption{Smoothness Conditions II for a $G$ invariant symmetric $2\times 2$ tensor $T$}\label{smooth2}
In Table D recall that $\ell_{-1}'=\{\dot c, X\}$ and $\ell_j=\{Y_1,Y_2\}$.
\end{table}

\end{document}